\documentclass{amsart}
\usepackage[utf8]{inputenc}
\usepackage{amsmath, amssymb}
\usepackage{tikz}
\usetikzlibrary{cd}
\usepackage{enumitem}
\setlist[enumerate,1]{label=(\arabic*), ref=(\arabic*), itemsep=0em}

\usepackage[colorlinks, citecolor=blue]{hyperref}
\usepackage{todonotes}

\newcommand{\Hilb}{{\rm{Hilb}}}
\newcommand{\Slip}{{\rm{Slip}}}
\newcommand{\Satbar}{\overline{\rm{Sat}}}

\newcommand{\Aut}{{\rm Aut}}

\newcommand{\Spec}{{\rm{Spec}}\,}
\newcommand{\Supp}{{\rm{Supp}}\,}
\newcommand{\Hom}{{\rm{Hom}}}
\newcommand{\Ext}{{\rm{Ext}}}

\newcommand{\QQ}{{\mathbb Q}}
\newcommand{\PP}{{\bf P}}

\newcommand{\CC}{{\mathbb C}}

\newcommand{\ZZ}{{\mathbb Z}}
\newcommand{\OO}{{\mathcal{O}}}

\newcommand{\onto}{\twoheadrightarrow}

\DeclareMathOperator{\Id}{Id}

\def\Gr #1#2{{\mathbb G}\left(#1,#2\right)} 

\newcommand{\cI}{{\mathcal I}}

\newcommand{\Hvanilla}{\Hilb_{n}(\PP^{n-1})}

\newcommand{\Hgood}{\Hilb^{\rm{good}}_{n}(\PP^{n-1})}
\newcommand{\HgoodGor}{\Hilb^{\rm{good}, \rm{Gor}}_{n}(\PP^{n-1})}

\newcommand{\VPSsbl}{{\mathrm{VPS}}^{\mathrm{sbl}}(Q, H)}
\newcommand{\VPSuns}{{\mathrm{VPS}}^{\mathrm{uns}}(Q, H)}
\newcommand{\VPSgood}{{\mathrm{VPS}}^{\mathrm{good}}(Q, H)}

\newcommand{\VPS}{\mathrm{VPS}(Q, H)}
\newcommand{\VAPS}{\mathrm{VAPS}(Q, n)}
\newcommand{\VPSn}{\mathrm{VPS}(Q, n)}
\newcommand{\VPSf}{\mathrm{VPS}(Q, 4)}
\newcommand{\VSPn}{\mathrm{VSP}(Q, n)}

\newcommand{\Syz}{\mathrm{Syz}_n}%
\newcommand{\GL}{\mathrm{GL}}%
\newcommand{\PGL}{\mathrm{PGL}}%
\newcommand{\SL}{\mathrm{SL}}%
\newcommand{\sat}{\mathrm{sat}}%

\newtheorem{lemma}{\bf Lemma}[section]
\newtheorem{proposition}[lemma]{\bf Proposition}
\newtheorem{theorem}[lemma]{\bf Theorem}
\newtheorem{innercustomthm}{\bf Salvaged Theorem}
\newenvironment{customthm}[1]
  {\renewcommand\theinnercustomthm{#1}\innercustomthm}
  {\endinnercustomthm}
\newtheorem{corollary}[lemma]{\bf Corollary}

\newtheorem{example}[lemma]{\bf Example}

\newtheorem{remark}[lemma]{\bf Remark}

\title{The variety of polar simplices II}
\author[Joachim Jelisiejew]{Joachim Jelisiejew}
\address{Faculty of Mathematics\\ Informatics
and Mechanics\\ University of Warsaw\\
Banacha 2, Warsaw, Poland}
\email{j.jelisiejew@uw.edu.pl}

\author[Kristian Ranestad]{Kristian Ranestad}
\address{Matematisk institutt\\
         Universitetet i Oslo\\
         PO Box 1053, Blindern\\
         NO-0316 Oslo\\
         Norway}
\email{ranestad@math.uio.no}
\author[Frank-Olaf Schreyer]{Frank-Olaf Schreyer}
\address{Mathematik und Informatik\\
                    Universität des Saarlandes\\
                    D-66123 Saarbrücken\\
                    Germany}
\email{schreyer@math.uni-sb.de}

\thanks{JJ is supported by National Science Centre grant 2020/39/D/ST1/00132. This work is partially supported by  the Thematic Research Programme "Tensors: geometry, complexity and quantum entanglement", University of Warsaw, Excellence Initiative – Research University and the Simons Foundation Award No. 663281 granted to the Institute of Mathematics of the Polish Academy of Sciences for the years 2021-2023.}
\date{\today{}}

\newcommand{\kk}{k}
\begin{document}
\begin{abstract}
    We discuss the space $VPS(Q,H)$ of ideals with Hilbert function
    $H=(1,n,n, \ldots )$ that are apolar to a full rank quadric $Q$.
    We prove that its components of saturated ideals are closely related to the locus of
    Gorenstein algebras and to the Slip component in border apolarity. We also
    point out an important error in~\cite{RS} and provide the necessary
    corrections.
\end{abstract}

\maketitle%

\section{Introduction}

For a fixed nondegenerate quadric $q$ in $n$ variables, a polar simplex is a
tuple $[\ell_1], \ldots ,[\ell_n]\in \mathbb{P}^{n-1}$ of points, where
$\ell_i$ are such that $q = \ell_1^2 + \ldots + \ell_n^2$. The locus of all
polar simplices is a principal homogeneous space for the group $SO(q) \simeq SO_n$.
Informally speaking, the variety of polar simplices is a compactification of
this locus. As we discuss below, there is actually more than one possible compactification and
choosing the correct one is subtle. The variety is important for two
main reasons. First, it is a ``simplest'' example of a variety of sums of
powers~\cite{ranestad_schreyer_VSP}, which are useful as examples of special
projective varieties and for applications in tensors~\cite{Huang_Michalek_Ventura,
Gallet_Ranestad_Villamizar, Bolognesi_Massarenti, Ranestad_Voisin, Carlina_Catalisano_Oneto,
BB}. Second, as explained below, it serves as a mean to investigate the
Gorenstein locus of the Hilbert scheme of points: it has smaller dimension,
which is important for computations such as~\cite{Ilten} and at the same time
contains each abstract Gorenstein subscheme of degree $n$.

Let $S = \kk[x_1,  \ldots , x_n]$ be a polynomial ring which we view as a
homogeneous coordinate ring of the $\PP^{n-1}=\PP(S_1^{*})$. Fix a number $d$ and a nonzero  $f\in
S_d^{*}$, then $F=\{f=0\}\subset \check{\PP}^{n-1}$. A finite subscheme $\Gamma\subseteq \PP^{n-1}$ is \emph{apolar}
to $F$ if $[f]$ lies in the span of the $d$-uple reembedding of $\Gamma$ in $\PP(S_d^{*})$. Apolarity
may also be formulated algebraically, see~\S\ref{sec:apolarity}, by requiring
that $I_{\Gamma} \subseteq f^{\perp}$, the ideal of forms in $S$ that annihilates $f$ by differentiation.
This condition may be formulated for ideals in general; an ideal $I\subset S$ is {\em apolar} to $F$ if $I \subseteq f^{\perp}$.

Consider the special case $d = 2$ and $F = Q=\{q=0\}$ being a full rank quadric.
Let $\VPSn\subseteq \Hvanilla$ denote the closure of the locus of
degree $n$ zero-dimensional schemes $\Gamma \subseteq \PP^{n-1}$
satisfying $I_{\Gamma} \subseteq q^{\perp}$. The
result~\cite[Corollary~2.2]{RS} and \cite[Proposition~6.5]{Skjelnes_Stahl}
both state that every $\Gamma\in \VPSn$ is apolar to $Q$.
It is not so, as the following example shows.
\begin{example}\label{ex:nonsaturatedlimit,n=4}
Let $n=4$, let $q = y_1y_4 +
y_2y_3$, where $y_i$ is the dual to $x_i$, and consider the ideal
\[
    I = (x_3x_1 - x_2^2,\, x_2x_3 - x_1x_4,\, x_2x_4,\, x_3x_4,\, x_4^2,\, x_3^2).
\]
This ideal is saturated and has Hilbert function $H_{S/I}(i) = 4$ for every
$i\geq 1$, hence $I = I_{\Gamma}$ for a finite, degree $4$ subscheme $\Gamma
\subseteq \PP^3$; in particular $[\Gamma]\in \VPSf$. Consider now a
$\kk^*$-action on $\PP^3$ corresponding to the grading by $(0, 1, 0, 1)$. The
quadric $Q$ is a semi-invariant for this action, so for every $\lambda\in
\kk^*$, the subscheme $\Gamma_{\lambda} := \lambda\cdot \Gamma$ corresponds to
a point $[\Gamma_{\lambda}]$ {in} $\VPSf$ {with} ideal given by
\[
    I(\Gamma_\lambda) = (x_3x_1 - \lambda^2x_2^2,\, \lambda(x_2x_3 - x_1x_4),\,
    \lambda^2x_2x_4,\, \lambda x_3x_4,\,
    \lambda^2x_4^2,\, x_3^2).
\]
The limit of those ideals at $\lambda\to 0$, taken degree by degree, is the ideal
\[
    I_{\lim} := (x_{2}^{4},\,x_{4}^{2 },\,x_{2}x_{4},\,x_{3}x_{4},\,x_{1}^{2}x_{4
    },\,x_{2}x_{3}-x_{1}x_{4},\,x_{3}^{2},\,x_{1 }x_{3}),
\]
whose saturation is $I_{\lim}^{\sat} = (x_3, x_4, x_2^4)$. Let $\Gamma_0 =
V(I_{\lim}^{\sat})$, then $\Gamma_0$ is a limit of $\Gamma_{\lambda}$ for $\lambda\to
0$, so $[\Gamma_0]\in \VPSn$ while $I_{\Gamma_0}=I_{\lim}^{\sat}$ is not contained in $q^\perp$, so $\Gamma_0$ is not apolar to
$Q$.
\end{example}
The example implies in particular that several of the main results in
\cite{RS} are wrong as stated.
The purpose of this article is to both discuss what can salvaged and how and,
which is of independent interest, show unexpected connections between
$\VPSn$ and the moduli of Gorenstein algebras.

Weronika Buczy\'{n}ska and Jaros{\l}aw Buczy\'{n}ski introduced the multigraded
Hilbert scheme into the subject of apolarity and
observed that it is better than the usual Hilbert scheme also when one
considers VPS~\cite[\S7.6]{BB}.

The multigraded Hilbert scheme $\Hilb^H$, defined
in~\cite{Haiman_Sturmfels__multigraded}, parametrizes homogeneous ideals with a given Hilbert function $H$.
For a fixed form $f$, the condition   $I\subset f^\perp$ for an
ideal $I$ is a \emph{closed} condition in the multigraded Hilbert scheme.
Example~\ref{ex:nonsaturatedlimit,n=4} above shows that the apolarity condition $I_{\Gamma}\subset
f^{\perp}$ is \emph{not closed} in the usual Hilbert scheme. Hence, it is more
natural to work in the multigraded Hilbert scheme.

We consider the Hilbert function $H:=(1,n,n,\ldots )$ and, for the full rank
quadric $Q$, define
$$\VPS=\{I\in \Hilb^H| I\subset q^\perp\}\subset \Hilb^H.$$
This is a closed subscheme of $\Hilb^H$.
Let $\VPSgood\subseteq \VPS$ be the locus of saturated ideals. This
locus is
open by~\cite[Theorem~2.6]{JM}. Let $\VPSsbl$ be its closure, so $\VPSsbl$ is
a union of irreducible components of $\VPS$. Let $\VPSuns = \VPS
\setminus \VPSgood$ with its reduced scheme structure.
In analogy with the locus of smoothable schemes in the Hilbert scheme, we say that $\VPSsbl$  is the locus of ideals that are \emph{{\bf S}atura{\bf BL}e}, this is consistent with the notation of~\cite{JM} thanks to
Corollary~\ref{ref:saturableAndApolarImpliesSaturableInApolar:cor}.
The intersection $\VPSsbl\cap\VPSuns$ is the locus of {\em unsaturated limit
ideals}: apolar unsaturated ideals that are limits of saturated apolar ideals.

Associating to each ideal $I\in \VPS$ the space $I_2$ of quadrics in the ideal defines a forgetful map
$$\pi_G:\VPS\to \Gr{\binom{n}{2}}{q^\perp_2},$$
into the Grassmannian of $\binom{n}{2}$-dimensional subspaces in $q^\perp_2$.
We denote
\begin{align*}
    \VPS_G:=& \pi_G(\VPS)\\
    \VPSgood_G:=& \pi_G(\VPSgood)\\
    \VPSsbl_G:=& \pi_G(\VPSsbl)\\
    \VPSuns_G:=& \pi_G(\VPSuns).
\end{align*}
In this paper we will be concerned with these loci.
The map $\pi_G$ restricted to $\VPSgood$ is
an isomorphism (see Proposition~\ref{prop:goodlocusInGrassmannian}),
while the restriction to the other loci, in general, is not.

The scheme $\VPS$ is different from $\VPSsbl$ in general, and similarly for $\VPS_G$ and $\VPSsbl_G$. For $n\leq 3$ the schemes
$\VPS$ and $\VPSgood$ coincide
and are isomorphic to $\VPS_G$, see Proposition~\ref{prop:regularity}. As soon as $n\geq 4$, both the
complement of $\VPSgood$ in $\VPSsbl$, i.e. the intersection $\VPSsbl\cap\VPSuns$, and the complement of $\VPSsbl$ in $\VPS$ {are} nonempty.
For $n=4$ the ideals of the intersection $\VPSsbl\cap\VPSuns$ have an unexpected beautiful connection to
the geometry of the inverse quadric $Q^{-1}=\{q^{-1}=0\}$. The latter defines the collineation $q^{-1}:S^*_1\to S_1$ inverse to the symmetric collineation defined by $q:S_1\to S^*_1$.

\begin{example}{\rm
    For $n=4$, the loci are summarized in Figure~\ref{figure:nFour}.
    In this case, $\VPSuns$ is given by ideals $I_\Gamma\cap q^\perp$,
    where $\Gamma \subseteq L$ is finite of length $4$ and $L$ is a line in $\PP^3$.
    The intersection $\VPSsbl\cap \VPSuns$ is $5$-dimensional and consists of ideals $I_\Gamma\cap q^\perp$ where
    $\Gamma \subseteq L$ and $L$ is a line in the quadric
    $Q^{-1}$ inverse to $Q$, see
    Section~\ref{section:unsat}.}
\end{example}

For $n\geq 6$ the locus $\VPSgood \subseteq \VPSsbl$ is singular (because
there are obstructed
Gorenstein schemes of degree $n$, see Theorem~\ref{ref:introSmoothMap:thm}
below). For $n\leq 3$, the two loci coincide, are isomorphic to $\VPS_G$, and are smooth by Proposition~\ref{prop:regularity} and \cite[Proposition 10]{Mukai}. We prove:
\begin{theorem}[{Corollary~\ref{ref:VSPborderSmooth:cor},
    Corollary~\ref{ref:Fano:cor}}]\label{ref:introSmoothness:thm}
    For $n=4,5$ both $\VPSsbl$ and $\VPSsbl_G$ are smooth.
\end{theorem}
There is a map from the multigraded Hilbert scheme to the Hilbert scheme. For
$n\leq 13$ it sends the scheme $\VPSsbl$ to the variety of sums of powers $\VSPn$. 
For $n\geq 4$ we do not know whether $\VSPn$ is smooth; as noted above, the Hilbert scheme
compactification seems to be not as interesting as the other two: being apolar is not a
closed property in the Hilbert scheme so $\VSPn$ has no functorial interpretation, see
Remark~\ref{ref:VSPinHilb:remark} for a bit more discussion.

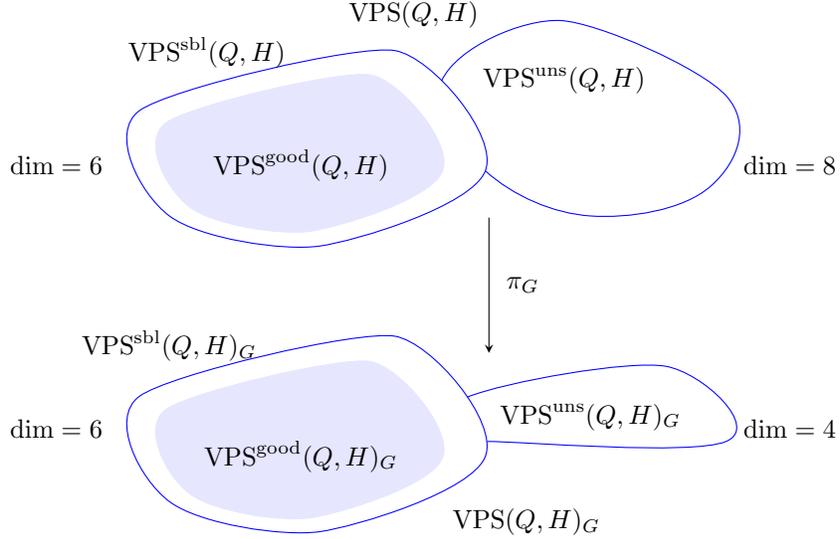
\begin{figure}\label{figure:nFour}
\noindent\begin{tikzpicture}[x=0.5cm,y=0.5cm]
    \draw[blue](0.703982,4.444432) .. controls (1.510387,5.0357957) and
    (6.51877,6.215773) .. (7.503982,6.044431950460819) .. controls (8.489194,5.8730907)
    and (10.262961,3.7777777) .. (9.903982,2.8444319) .. controls (9.545003,1.9110863)
    and (6.4939313,0.9858533) .. (5.503982,0.84443194) .. controls (4.5140324,0.7030106)
    and (2.303982,1.0444319) .. (1.503982,1.644432) .. controls
    (0.703982,2.244432) and
    (-0.102423,3.8530684) .. (0.703982,4.444432);
    \draw[fill, blue!10,xscale=0.8, yscale=0.8,xshift=15,yshift=10](0.703982,4.444432) .. controls (1.510387,5.0357957) and
    (6.51877,6.215773) .. (7.503982,6.044431950460819) .. controls (8.489194,5.8730907)
    and (10.262961,3.7777777) .. (9.903982,2.8444319) .. controls (9.545003,1.9110863)
    and (6.4939313,0.9858533) .. (5.503982,0.84443194) .. controls (4.5140324,0.7030106)
    and (2.303982,1.0444319) .. (1.503982,1.644432) .. controls
    (0.703982,2.244432) and
    (-0.102423,3.8530684) .. (0.703982,4.444432);
    \draw[blue] (8.743982,5.244432) .. controls (9.151196,6.1388593) and
    (10.905364,6.896991) .. (11.903982,6.844431950460819) .. controls
    (12.9026,6.791873) and (15.6313095,5.584372) .. (16.303982,4.844432) ..
    controls (16.976654,4.1044917) and (16.651392,3.108796) ..
    (15.903982,2.444432).. controls (15.156572,1.7800682) and
    (13.701767,1.5779129) .. (12.703982,1.644432) .. controls
    (11.706197,1.710951) and (10.647276,2.1754673) .. (9.903982,2.8444319);
    \draw[-stealth] (10, 1.6) -- (10, -2);
    \draw[blue] (0.703982,-3.1555681) .. controls (1.510387,-2.5642045) and
    (6.51877,-1.3842269) .. (7.503982,-1.5555680495391813) .. controls
    (8.489194,-1.7269093) and (10.262961,-3.8222225) .. (9.903982,-4.755568)
    .. controls (9.545003,-5.688914) and (6.4939313,-6.6141467) ..
    (5.503982,-6.755568) .. controls (4.5140324,-6.8969893) and
    (2.303982,-6.555568) .. (1.503982,-5.955568) .. controls
    (0.703982,-5.355568) and (-0.102423,-3.7469318) .. (0.703982,-3.1555681);
    \draw[fill, blue!10,xscale=0.8, yscale=0.8,xshift=15,yshift=-17.5] (0.703982,-3.1555681) .. controls (1.510387,-2.5642045) and
    (6.51877,-1.3842269) .. (7.503982,-1.5555680495391813) .. controls
    (8.489194,-1.7269093) and (10.262961,-3.8222225) .. (9.903982,-4.755568)
    .. controls (9.545003,-5.688914) and (6.4939313,-6.6141467) ..
    (5.503982,-6.755568) .. controls (4.5140324,-6.8969893) and
    (2.303982,-6.555568) .. (1.503982,-5.955568) .. controls
    (0.703982,-5.355568) and (-0.102423,-3.7469318) .. (0.703982,-3.1555681);
    \draw[blue] (9.423982,-3.1755681) .. controls (10.427059,-2.7709527) and
    (13.719198,-2.1817827) .. (14.703982,-2.355568049539181) .. controls
    (15.688766,-2.5293534) and (17.227058,-3.9709527) .. (16.303982,-4.355568)
    .. controls (15.380905,-4.7401834) and (10.903982,-4.355568) ..
    (9.95,-4.355568);
    \node at (10.9, -0.2) {$\pi_G$};
    \node at (8, 7) {$\VPS$};
    \node at (11, -6.5) {$\VPS_G$};
    \node at (5, -4.75) {$\VPSgood_G$};
    \node at (5, 3) {$\VPSgood$};
    \node at (2.5, 6) {$\VPSsbl$};
    \node at (-1.5, 3) {$\dim = 6$};
    \node at (18, 3) {$\dim = 8$};
    \node at (-1.5, -4) {$\dim = 6$};
    \node at (18, -4) {$\dim = 4$};
    \node at (1.5, -1.8) {$\VPSsbl_G$};
    \node at (12, 5.25) {$\VPSuns$};
    \node at (12.7, -3.7) {$\VPSuns_G$};
\end{tikzpicture}
\caption{Moduli spaces for $n=4$.}
\end{figure}

\subsection{Connections to the Gorenstein locus}
A saturated apolar ideal of a finite scheme is locally Gorenstein \cite[Proof of Proposition 2.2]{BB10}.
{Let $\Hgood\subset \Hvanilla$ be the locus of  $\Gamma\subseteq \PP^{n-1}$ which satisfy
$h^1(\cI_{\Gamma}(1)) = 0$.
The locus $\VPSgood\subset \VPSsbl$ of saturated ideals is closely connected to the locus $\HgoodGor$
of locally Gorenstein $\Gamma$ in $\Hgood$.}

\begin{theorem}[Proposition~{\ref{smoothMap}}]\label{ref:introSmoothMap:thm}
    The natural map
    \[
        \GL_{n}\times \VPSgood\to \Hgood;\quad (g,I)\mapsto g\cdot I
    \]
    is smooth and its image
    is $\HgoodGor$. In particular, the singularity types encountered on
    $\VPSgood$ and on $\HgoodGor$ coincide.
\end{theorem}
The bridge given by the theorem allows us to apply results from both sides.
First, we improve and make sharp the question about when $\VPSgood$ is
reducible, raised in~\cite[Theorem~1.3]{RS}.
\begin{corollary}[{Theorem~\ref{ref:reducibility:thm}}]
    The locus $\VPSgood$ is irreducible exactly when $n\leq 13$.
\end{corollary}
Second, we use the result~\cite[Corollary~5.16]{RS} to obtain properties of
the Gorenstein locus itself.
\begin{corollary}
    The Gorenstein locus of every $\Hilb_n(\PP^k)$ is reduced for every $n\leq
    6$.
\end{corollary}
\begin{proof}
    Reducedness of a point $[Z\subseteq \PP^k]\in \Hilb_n(\PP^k)$ depends only
    on the underlying $Z$, so $\Hilb_n(\PP^k)$ is reduced if and only if
    $\Hgood$ is reduced. By Theorem~\ref{ref:introSmoothMap:thm} this holds if
    and only if $\VPSgood$ is reduced. The reducedness of the latter is proven
    in~\cite[Corollaries~5.12 and 5.16]{RS}.
\end{proof}
Theorem~\ref{ref:introSmoothMap:thm} extends also to the whole ${\VPSsbl}$. The
locus of saturated ideals in $\Hilb^H$ is open and its closure is called
$\Satbar^H$. This closure is a union of components of $\Hilb^H$, for $n\leq 13$ it is a single
component, called $\Slip^H$. This locus has recently gained much attention thanks
to applications in border apolarity~\cite{BB, JM}. The theorem extends as
follows.
\begin{theorem}[Proposition~\ref{ref:smoothness}, Corollary~\ref{ref:saturableAndApolarImpliesSaturableInApolar:cor}, Corollary~\ref{ref:VSPborderSmooth:cor}]\label{ref:introSmoothMapWhole:thm}
    The natural map
    \[
        \GL_{n}\times {\VPSsbl}\to \Satbar^H;\quad (g,I)\mapsto g\cdot I
    \]
    is smooth.
    {For $n\leq 5$}, {$\VPSsbl$ is smooth and } the component $\Slip^H$ is smooth along the image of this map.
   
\end{theorem}
The component $\Slip^H$ is \emph{not} smooth in all points, not even for
$n=4$: the saturated ideal of a second-order neighbourhood of a point of $\PP^3$ yields a singular point of $\Slip^H$.

\subsection{Salvaging \cite{RS}}
In this section we report on the state of the results in the paper \cite{RS}.

Our notation differs from that of~\cite{RS}, we follow rather the notation
of~\cite{ranestad_schreyer_VSP}.
The paper~\cite{RS} uses ${\rm VPS}(Q,n)$ for the \emph{\textbf{V}ariety of
\textbf{P}olar \textbf{S}implices} as a subscheme of the Hilbert scheme;
what we call $\VSPn$ the \emph{\textbf{V}ariety of \textbf{S}ums of
\textbf{P}owers}. The
paper~\cite{RS} uses $\VAPS$ for the image in the Hilbert scheme of the scheme $\VPSsbl$. 

The crucial error is found in the proof of \cite[Corollary 2.2]{RS} that wrongly
claims that $${\rm VPS}^{\rm good}(Q,n) = {\rm VPS}(Q,n),$$ while we only have an open immersion $${\rm VPS}^{\rm good}(Q,n) \subseteq {\rm VPS}(Q,n).$$
We therefore discuss the statements that depend on this error.

\subsubsection{Introduction.} The second part of \cite[Theorem~1.1]{RS} asserts
that for $n\geq 6$ the variety $\VSPn$ is
singular, rational and $\binom{n}{2}$-dimensional. This is true (after
possibly reducing) with the same proof. The first part of Theorem~1.1 claims
that for $2\leq n\leq 5$ the variety $\VSPn$ is additionally smooth of Picard
rank $1$ and is Fano of index $2$. For $n=2,3$ we have $\VPS=\VPSsbl$ and $\pi_G$ is an isomorphism onto the image $\VPS_G$.  In particular both are isomorphic to $\VSPn$ and
the argument of~\cite{RS} is correct.
For $n=4,5$, we could ask the same question for $\VPS$ and $\VPSsbl$ and their images under $\pi_G$.
In both cases, by Theorem~\ref{ref:introSmoothness:thm}, $\VPSsbl$ 
is smooth and admits a birational morphism onto the smooth $\VPSsbl_G$, that contracts the boundary divisor $\VPSsbl\setminus \VPSgood$ (cf. Remark \ref{contraction}),  hence the
Picard rank of $\VPSsbl$ is at least two. For $\VSPn$ we do not know whether it is smooth,
but if it were,
see
Remark~\ref{ref:VSPinHilb:remark}, its Picard rank would also be at least two.
When replacing $\VPSsbl$ by the Grassmannian subscheme $\VPSsbl_G$, however, we salvage also the first part of \cite[Theorem~1.1]{RS}.
\begin{customthm}{1.1}[{Corollary~\ref{ref:Fano:cor}}]
    For $2\leq n\leq 5$ the Grassmannian subscheme $\VPSsbl_G$ is a smooth rational
    $\binom{n}{2}$-dimensional Fano variety of index $2$ and Picard number
    $1$.
\end{customthm}

The theorem \cite[Theorem~1.2]{RS} concerns $\VPSsbl_G$, the Grassmannian subscheme. We
were able to obtain the result of this theorem for $n=4$, with a correct degree, using a more nuanced machinery
of excess intersections. The case $n=5$ remains open. We refer to~\cite{RS}
for the notation regarding the Gauss map; we will not use it in the present
article.
\begin{customthm}{1.2}[{Proposition~\ref{ref:linearSection:prop}}]
    The variety $\VPSsbl_G$ contains the image $TQ^{-1}$ of the Gauss map. When
    $n=4$ the restriction of the Pl\"ucker line bundle generates the Picard
    group of $\VPSsbl_G$ and the degree is $362$.
\end{customthm}
The theorem \cite[Theorem~1.3]{RS} concerns the linear span of $\VPSsbl_G$, and is wrong.  The image $\VPSuns_G\cap \VPSsbl_G$ of unsaturated limit ideals does not lie in the span of $TQ^{-1}$ (loc.cit.).  Whether $\VPSsbl_G$ is a linear section of the 
Grassmannian therefore remains an open problem.  It is true for $n=3$, and we give a computational proof for $n=4$.

Finally, \cite[Proposition~1.4]{RS} remains correct with the same proof.
\subsubsection{Sections~2, 3, 4,5.}
The results of these sections are local and the proofs are not effected by
mistakes concerning the compactifications up to restricting to the locus of
linearly normal schemes. That is, up to replacing $\VAPS$ and ${\rm VPS}(Q,n)$ by 
$\VPSgood$ and its smoothable component, the statements are correct according to our
knowledge, except the final part of Corollary~2.2.

\subsubsection{Section 6}  The theorem \cite[Theorem~6.3]{RS} concerns the degree of the Grassmannian subscheme $\VPSsbl_G$.  The presence of unsaturated limit ideals in $\VPSsbl$ means that the claim of the theorem is wrong, instead the degree formula is a contribution in a computation of the degree of $\VPSsbl_G$ using excess intersection.
We show this in the case $n=4$, see Proposition~\ref{ref:linearSection:prop} and Remark \ref{excess}.

{ \subsection{Legend} In section \ref{sec:apolarity} we show that the locus $\VPSgood$ of saturated apolar ideals is irreducible if and only if the quadric $Q$ has rank at most $13$.  In the following section \ref{section:unsat}, we discuss the unsaturated ideals in the boundary $\VPSsbl\setminus \VPSgood$ and give a precise characterization when $n=4,5$. In section \ref{globalproperties} we give our results on global properties both of the multigraded Hilbert scheme compactification 
$\VPSsbl$ and of the Grassmannian compactification $\VPSsbl_G$ of $\VPSgood$.  The appendices \ref{sec:computerPart} and \ref{vsp4package} contain computer code \cite{Mac2S} in Macaulay2 \cite{M2},  used in our computational arguments and a result on $\VPSsbl_G$ {for} $n=4$, for which we only give a computational proof.
}
\subsection{Notation}
We let $k$ be a field of characteristic zero. 
Our computations, see Appendices \ref{sec:computerPart} an \ref{vsp4package}, are performed over $\QQ$. 
Let $$S = \kk[x_1, \ldots, x_n]\quad {\rm and} \quad T = \kk[y_1, \ldots ,y_n]$$ be polynomial
rings. We view $x_1, \ldots ,x_n$ and $y_1, \ldots ,y_n$ as dual bases of dual
spaces $S_1$ and $T_1$.  Differentiation defines bilinear pairings $S_e\times T_d\to T_{d-e}, d\geq e$ which induces isomorphisms $S_d\to T_d^*$ for any $d{\geq}0$.  In this sense we say that $T$ is dual to $S$.
We let $S$ be the coordinate ring of $\PP^{n-1}$, and $T$ be the coordinate ring of the dual space $\check{\PP}^{n-1}$.  Thus we may set $\PP^{n-1}=\PP(T_1)$ the projective space of $1$-dimensional subspaces in $T_1$.

For a homogeneous polynomial $f\in T_d\cong S_d^*$ let $f^{\perp} \subset S$ denote its apolar ideal.
For a subscheme $\Gamma \subset \PP^{n-1}$ let $I_{\Gamma}\subset S$ denote is homogeneous ideal; note that this ideal is saturated.
We say that an ideal $I\subseteq S$ is \emph{apolar} to $F(=\{f=0\}$ if $I\subseteq
f^{\perp}$. We say that a zero-dimensional $\Gamma \subset \PP^{n-1}$ is
\emph{apolar} to $F$ if its homogeneous ideal is apolar, that is, if $I_{\Gamma} \subset f^{\perp}$. 
We usually denote quadratic forms in $T$ by $q$; it defines a quadric $Q = \{q=0\} \subset \check\PP^{n-1}$ and a collineation: $q:S_1\to T_1$. When $q$ is nondegenerate, then denote by $q^{-1}\in S_2$ the quadratic form that defines the inverse collineation,
$q^{-1}:T_1\to S_1,$ and by $Q^{-1}=\{q^{-1}=0\}\subset \PP^{n-1}$ the corresponding quadric.

We also identify $q\in T_2$ with the
associated map $q\colon S_2\to k$.
If $q\in T_2$ is a full rank quadric, then $q^\perp\subset S$ is generated by degree two
elements and $q^\perp_2 = \ker(q\colon S_2\to k)$, so an ideal $I$ is apolar to $Q$  if and only if $I_1=0$ and
$I_2\subset \ker(q\colon S_2\to k)$. A subscheme $\Gamma$ as above is apolar to $Q$ if and only if its homogeneous ideal $I_{\Gamma}$ is apolar.  

{
\subsection*{Acknowledgements}
We thank Tomasz Ma{\'n}dziuk and Emanuele Ventura for pointing out an issue in an earlier version of the paper, and anonymous referees for helpful comments on presentation.
}

\section{Good points of VPS}\label{sec:apolarity}

The following result essentially appears in~\cite{RS}, but due to its
importance we provide a full proof.
\begin{proposition}[{\cite[Lemma~2.6]{RS}}]\label{ref:charOfPossiblyApolar:prop}
    Let $\Gamma \subset \PP^{n-1}$ have length $n$. The following are
    equivalent
    \begin{enumerate}
        \item\label{it:locGorenstein} $\Gamma\subset \PP^{n-1}$ is (locally) Gorenstein and linearly normal,
        \item\label{it:quadricExists} there exists a full rank quadric $Q$ such that $\Gamma$ is
            apolar to $Q$.
    \end{enumerate}
\end{proposition}

\begin{proof}
    Let $A = H^0(\OO_{\Gamma})$.
    Suppose~\ref{it:locGorenstein} holds, so the dualizing $A$-module $\omega_A
    = \Hom_A(A, k)$ is isomorphic to $A$. Let $\varphi\in \omega_A$ be its generator,
    then the pairing $A\times A\to k$ given by $(a_1, a_2)\mapsto
    \varphi(a_1a_2)$ is symmetric and perfect. Since $\Gamma\subset \PP^{n-1}$ is linearly
    normal, the restriction of linear forms $S_1\to H^0(\OO_{\Gamma}(1))
    \simeq A$ is an isomorphism and we obtain a symmetric perfect pairing $S_1\times
    S_1\to k$. Let $q\in T_2$ be the corresponding quadric. It has full rank.
    Moreover, it arises from the following commutative diagram
    \[
        \begin{tikzcd}
            & S_1 \times S_1 \ar[d]\ar[r] & H^0(\OO_{\Gamma}(1)) \times
            H^0(\OO_{\Gamma}(1)) \ar[d] \ar[r, "\simeq"] & A\times A\ar[d] \\
            (I_{\Gamma})_2 \ar[r, hook] & S_2 \ar[r] & H^0(\OO_{\Gamma}(2)) \ar[r, "\simeq"] & A
            \ar[r, "\varphi"]& k
        \end{tikzcd}
    \]
    so $q^\perp \supset (I_{\Gamma})_2$, hence $\Gamma$ is apolar to
    $Q$.

    Assume~\ref{it:quadricExists}. Consider the diagram as above, where
    $\varphi\colon A\to k$ is the functional induced by $q\colon S_2\to k$.
    Since $q$ has full rank, the natural map $S_1\to A$ has to be injective,
    so an isomorphism by comparing dimensions. This shows that $A \times A\to
    k$ from the diagram is a perfect pairing, hence $\omega_A$ is isomorphic
    to $A$, hence $\Gamma$ is locally Gorenstein.
\end{proof}

Recall that $\Hilb^H$ is the multigraded Hilbert scheme parameterizing ideals with Hilbert function $H = (1,n,n,n, \ldots )$.
Let
$\VPSgood\subseteq \Hilb^H$ denote the locus of ideals which are saturated and
apolar to $Q$. Recall that we have a natural map $\Hilb^H\to
\Hilb_n(\PP^{n-1})$.
Let $\Hgood\subseteq \Hilb_n(\PP^{n-1})$ denote the open
locus of $\Gamma$ with $H^1(I_{\Gamma}(1)) = 0$; equivalently those are the
subschemes which span $\PP^{n-1}$, or, in yet other words, the subschemes
with $H_{S/I_{\Gamma}} = H$.
\begin{proposition}\label{ref:twoEmbeddings:prop}
    The locus $\VPSgood$ is locally closed in $\Hilb^H$. Under the map
    $\Hilb^H\to \Hilb_{n}(\PP^{n-1})$ it maps isomorphically onto the
    locus of apolar subschemes in $\Hgood$.
\end{proposition}
\begin{proof}
    Let $\Hilb^{H, \sat}$ inside $\Hilb^H$ denote the subset of saturated ideals,
    this is an open subset~\cite[Theorem~2.6]{JM} and so comes with a natural scheme structure.
    The condition $I\subseteq q^\perp$ is closed in $\Hilb^H$ and so
    $\VPSgood$ is closed in $\Hilb^{H,\sat}$.
    By~\cite[Theorem~3.9]{JM} the subset
    $\Hilb^{H, \sat}$ maps isomorphically to $\Hgood$ under the natural map
    $\Hilb^H\to \Hilb_n(\PP^{n-1})$. Restricting this isomorphism to
    $\VPSgood$, we get the final claim.
\end{proof}

We now describe the topology of $\VPSgood$.
\begin{proposition}\label{connectednessOfVPSgood}
    The space $\VPSgood$ is connected.
\end{proposition}

\begin{proof}
    This follows from the argument in~\cite[Section~6]{RS} which is correct
    if one restricts from $\VPS$ to $\VPSgood$.
    Essentially the same argument is given independently
    in~\cite[Proposition~4.1]{hoyois2021hermitian}: any $k$-point $\Gamma$ of
    $\VPSgood$ degenerates to a fixed one and this degeneration is possible
    inside $\VPSgood$ because it preserves \emph{orientation} of $\Gamma$, as
    defined in \cite[2.1]{hoyois2021hermitian}.
\end{proof}

Before we give more refined information about the topology, we need a
technical idea of orthogonalization, which links $\VPSgood$ and $\Hgood$.
\subsection{Orthogonalization}

The following key technical theorem shows that any
infinitesimal deformation of an apolar ideal can be ``orthogonalized'' to a
deformation in $q^\perp$. In essence, it says the following: suppose that
$I'\subseteq q^\perp$ is an apolar ideal and $I$ is its
deformation (in the multigraded sense) over a base with a nilpotent $\varepsilon$ so $I|_{\varepsilon=0} = I'$.
It may happen that $I$ does not lie in $q^\perp$, for
example when $I$ comes from a infinitesimal change of coordinates that moves
$q$. The theorem says that this is essentially the only possibility: there exists a
coordinate change $g = g(\varepsilon)$ so that $g\circ I$ is contained in the
ideal $q^\perp$.
It follows from the fact that the $\GL_n$-action on full
rank quadrics is transitive in the infinitesimal neighbourhood of $q$.

Let us introduce the setup. For a $k$-algebra $A$ let $S_A :=
S\otimes_{k} A \simeq A[x_1, \ldots ,x_n]$. For a surjection of $k$-algebras
$A\onto A'$ we get a surjection $S_{A}\onto S_{A'}$ which in down-to-earth
terms reduces the
coefficients of the polynomials modulo $J = \ker(A\to A')$.
Let $T$ be a polynomial ring dual to $S$ and let $q\in T_2 \subseteq (T_A)_2$
be a full rank quadric, as in the setup at the beginning of~\S\ref{sec:apolarity}. Let $q^\perp\subset S_A$ be its apolar ideal. We denote
by $q^\perp$ also its image in $S_{A'}$.
\begin{theorem}[orthogonalization]\label{thm:orthogonalization}
    Let $A\onto A'$ be a surjection of finite local $k$-algebras with kernel $J$ satisfying $J^2
    = 0$. Let $I' \subseteq S_{A'}$, $I\subseteq S_A$ be graded vector
    subspaces.
    Assume that
    \begin{enumerate}
        \item $I'_2 \subseteq q^\perp$,
        \item $(S_A/I)_2$ is a free $A$-module,
        \item the image of $I_2$ in $(S_{A'})_2$ is $I'_2$.
    \end{enumerate}
    Then there exists an invertible matrix $g\in \GL_n(A)$ such that $g\mod J =
    \Id_n$ and $(g\circ I)_2 \subseteq q^\perp$.
\end{theorem}

\newcommand{\qtilde}{\widetilde{q}}
\begin{proof}
    The duality of $S$ and $T$ induces a duality of free $A$-modules $(S_A)_2$
    and $(T_A)_2$ and, as a consequence, the free $A'$-modules $(S_{A'})_2$
    and $(T_{A'})_2$.
    The space perpendicular to $q\in (T_A)_2$ is
   $q^\perp\subset (S_A)_2$.
    Since $S_{A'}/q^\perp$ is a free $A'$-module of rank one, the surjections 
    $$(S_{A'})_2\onto (S_{A'}/I)_2 \onto S_{A'}/q^\perp$$ dualize
    to a composition $A'\to (S_{A'}/I)_2^{\vee} \to (T_{A'})_2$ which sends
    $1$ to $q$.

    The $A$-modules $(S_A/I)_2$, $(S_A/q^\perp)_2$ are free, hence
    \[
        \Hom_{A}\left(\left(\frac{S_A}{I}\right)_2,
        \left(\frac{S_A}{q^\perp}\right)_2\right)\otimes A' \simeq
        \Hom_{A'}\left(\left(\frac{S_{A'}}{I'}\right)_2,
        \left(\frac{S_{A'}}{q^\perp}\right)_2\right).
    \]
    Since $I'_2 \subseteq q^\perp$,
    the right-hand side contains the natural surjection $(S_{A'}/I')_2\onto
    (S_{A'}/q^\perp)_2$. Let us lift it to a surjection $\varphi\colon
    (S_A/I)_2\to (S_A/q^\perp)_2$. The surjections
    \[
        \begin{tikzcd}
            (S_A)_2 \ar[r] & \left(\frac{S_A}{I}\right)_2 \ar[r, "\varphi"] &
            \left(\frac{S_A}{q^\perp}\right)_2
        \end{tikzcd}
    \]
    dualize to
    \begin{equation}\label{eq:twistedQuadric}
        A \to \left(\frac{S_A}{I}\right)^{\vee}_2 \to (T_A)_2.
    \end{equation}
    Let $\qtilde\in (T_A)_2$ denote the image of $1\in A$. Since the sequences
    for $I$ reduce to those for $I'$ modulo $J$, the quadric $\qtilde$
    satisfies $\qtilde \equiv q\mod J$. Let us identify quadrics with symmetric matrices.
    For an element 
    $g\in \GL_n(A)$,
    the action of $g$ on~\eqref{eq:twistedQuadric} yields
    \begin{equation}\label{eq:liftedQuadric}
        A \to \left(\frac{S_A}{g\circ I}\right)^{\vee}_2 \to (T_A)_2
    \end{equation}
    which maps $1\in A$ to the quadric $g\circ \qtilde$.
    We claim that there exists a $g$ such that $g\mod J = \Id_n$ and $g\circ \qtilde = q$.
    We view the vector space $T_2$ as an algebraic variety.
    Since $q$ has full rank, 
    the orbit map $\GL_n\to T_2\colon (h\mapsto h\circ q)$ is smooth. The quadric $\qtilde$ can be interpreted as a map
    $\qtilde\colon \Spec(A)\to T_2$. Let $\Spec(A')\to \GL_n$ be the constant map to the
    identity matrix. Since $\qtilde \equiv q\mod J$, we have a commutative diagram
    \[
        \begin{tikzcd}
            \Spec(A') \ar[d, hook]\ar[r] & \GL_n\ar[d, "\mathrm{smooth}"]\\
            \Spec(A) \ar[r, "\qtilde"] & T_2
        \end{tikzcd}
    \]
    From the infinitesimal lifting criterion~
\cite[{\href{https://stacks.math.columbia.edu/tag/02H6}{Tag 02H6}}]{stacks-project}, there is a lift $$
    \Spec(A)\to \GL_n\colon 1\mapsto h, \;{\rm with}\; h\circ q = \qtilde.$$ Taking $g = h^{-1}$ yields
    the desired element such that $g\circ \qtilde = q$.
    Dualizing~\eqref{eq:liftedQuadric} we get that $(g\circ I)_2 \subseteq
    q^\perp$.
\end{proof}

For a moment, we change focus from $\Hilb^H$ to $\Hgood\subseteq \Hvanilla$. By
Proposition~\ref{ref:twoEmbeddings:prop} the reader can choose one of two
perspectives: either think about $\Hgood$ or about the saturated locus of
$\Hilb^H$.
\begin{proposition}\label{smoothMap}
    The action map $\GL_{n} \times \VPSgood\to \Hgood$ is smooth and its
    image is the Gorenstein locus in $\Hgood$.
\end{proposition}

\begin{proof}
    We view $\VPSgood$ as a locally closed subset of $\Hvanilla$, using
    Proposition~\ref{ref:twoEmbeddings:prop}.
    Let $\mathcal{U}\to\VPSgood$ be the restriction of the universal
    family. The points of $\VPSgood$ are linearly normal
    by Proposition~\ref{ref:charOfPossiblyApolar:prop}.
    By Cohomology and Base Change, the $\OO_{\VPSgood}$-module
    \[
        \left(\frac{S\otimes_{\CC}
        \OO_{\VPSgood}}{I_{\mathcal{U}}}\right)_2
    \]
    is locally free and its fiber over a point $[\Gamma]$ is $(S/I_{\Gamma})_2$.
    The claim now follows from Theorem~\ref{thm:orthogonalization} and the infinitesimal lifting criterion~\cite[\href{https://stacks.math.columbia.edu/tag/02H6}{Tag 02H6}]{stacks-project}.
\end{proof}

\begin{remark}
    As explained in the introduction,
    the result of~\cite[Corollary~5.16]{RS} yields interesting news about the
    point $\Spec(S/(y_1^2 + y_2^2 + y_3^2 + y_4^2)^{\bot})$, the so-called G-fat
    point, in $\Hilb_{6}^{\mathrm{Gor}}(\PP^{5})$.
\end{remark}

\begin{theorem}\label{ref:reducibility:thm}
    The scheme $\VPSgood$ is irreducible,
  when $n\leq 13$.
\end{theorem}

\begin{proof}
    First we show that $\Hgood$ is irreducible.  For this, note that 
    the deformation theory of a zero-dimensional
    subscheme is independent of its projective embedding, see for
    example~\cite[p.4]{artin_deform_of_sings}. The
    Gorenstein locus is connected,
    see~\cite[Proposition~4.1]{hoyois2021hermitian}
    or~\cite{casnati_notari_irreducibility_Gorenstein_degree_9}, and linear normality condition on the Gorenstein locus in $\Hgood$ is open in the Gorenstein locus $\Hvanilla$. 
    So when the latter is irreducible, so is the former.  But the latter is irreducible when $n\leq 13$ by \cite[Theorem~A]{cjn13}. 

    Finally we prove that $\VPSgood$ is irreducible if and only if $\Hgood$ is irreducible. 
     We observe
    that both spaces are connected (see
    Proposition~\ref{connectednessOfVPSgood}) and related by smooth maps
    \[
        \VPSgood \leftarrow \GL_n \times \VPSgood \to \HgoodGor,
    \]
    see Proposition~\ref{smoothMap}, hence one of them contains a point which
    lies on two irreducible components if and only if the other does.
\end{proof}

\section{Unsaturated limit ideals}\label{section:unsat}

Having discussed $\VPSgood$,
in the following we analyse the boundary $\VPSsbl \setminus \VPSgood$. We will
see that the boundary is nonempty exactly
when $n\geq 4$ and give a necessary condition for a point of $\VPS$ to lie in
$\VPSsbl$. We will classify the points of the boundary for $n=4$ and for $n=5$. Two key inputs for the classifications are the results of~\cite{JM} and the orthogonalization, see Proposition~\ref{ref:smoothness}. An important geometric consequence of these, which is new, is the connection with the inverse quadric $Q^{-1}$, see Corollary~\ref{kri}.

We begin with a
general observation.
\begin{proposition}\label{boundarydivisors}
    Each component of the boundary $\VPSsbl \setminus \VPSgood$ is a divisor
    in $\VPSsbl$.
\end{proposition}
\begin{proof}
    Consider the natural projective map $\varphi\colon \Hilb^H\to \Hvanilla$. By
    Proposition~\ref{ref:twoEmbeddings:prop} the locus $\VPSgood$ maps
    isomorphically onto its image. Let
    \[
        X := \overline{\varphi(\VPSgood)}\subseteq
        \Hvanilla
    \]
    be the closure of the image, so that $X = \varphi(\VPSsbl)$.
    Let $\mathcal{D}\subseteq \Hvanilla$ be the locus of $\Gamma$ which are
    not linearly normal. Since $h^0(\OO_{\Gamma}(1)) =
    h^0(\OO_{\PP^{n-1}}(1))$, the locus $\mathcal{D}$ is a divisor. In fact, on the universal scheme $\tilde\Gamma\subset \Hvanilla\times \PP^{n-1}$, the map $H^0(\OO_{\PP^{n-1}}(1))|_{\tilde\Gamma}\to H^0(\OO_{\tilde\Gamma}(1))$ between vector bundles of rank $n$ drops rank on the divisor $\mathcal{D}$.
    By
    Proposition~\ref{ref:charOfPossiblyApolar:prop}, the intersection
    $X \cap \mathcal{D}$ lies in the complement of 
    $\varphi(\VPSgood)$ in $X$. Thus the
    boundary $X\setminus\varphi(\VPSgood)$ is a
    union of divisors of $X$, so also
    $\VPSsbl\setminus \VPSgood$ is a union of divisors.
\end{proof}

Now we pass to the more detailed description of the boundary.
\begin{proposition}\label{prop:regularity}
    Let $n\leq 5$ and let $[I]\in {\VPSsbl}$ be an unsaturated limit ideal. Then $\Gamma=V(I)$
    contains a length $4$ subscheme contained in a line. In particular, for
    $n\leq 3$ there are no unsaturated limit ideals. Moreover, under the same assumptions, the Hilbert function of $S/{I_\Gamma}$ is $(1,n-2,n-1,n,n,\ldots)$ or, for $n=5$, $(1,2,3,4,5,5,\ldots)$.
\end{proposition}

\begin{proof}
    Consider the Hilbert function $H$ of $S/(I_\Gamma)_2$. Since $I$ is
    unsaturated (in degree 2), $h^1(\mathcal{I}_{\Gamma}(2)) \neq 0$, so $H(2) \leq n-1$.
    From Macaulay's Growth Theorem, see~\cite[\S2E and Lemma~2.9]{cjn13}, it
    follows that
    \begin{enumerate}
        \item when $n\leq 3$, we get $H(m)\leq 2$ for all $m\geq 2$, hence a
            contradiction,
        \item when $n=4$, we get $H(m) = m+1$ for all $m\geq 2$, so
            $(I_{\Gamma})_2$ defines a line $L$ with $\Gamma \subset L$,
        \item when $n=5$, we get $H(m) = m+1$ or $H(m) = m+2$ for all $m\geq 2$, so
            $V((I_{\Gamma})_2)$ is a line or a subscheme with Hilbert polynomial $m+2$;
            in the latter case it is either a line $L$ with embedded point or a line $L$ with a
            disjoint point. In either case $\Gamma\cap L$ has length at least
            $n-1 = 4$.\qedhere
    \end{enumerate}
\end{proof}

The following is our main tool to discern $\VPSsbl$ from $\VPS$.
\begin{proposition}[{\cite[Example~4.2]{JM}}]\label{prop:necessaryForLimitOfSat}
Let $I$ be an unsaturated homogeneous ideal such that the Hilbert function of
$S/I$ is $(1,n,n,n, \ldots ,)$ while the Hilbert function of $S/I^{\sat}$ is
$(1,n-2, n-1, n, n,  \ldots )$.  Assume furthermore that $V(I_2)$ is a line
$L$ and some, possibly embedded, points (this is automatic when $n\leq 5$,
see Proposition~\ref{prop:regularity}).

If $I$ is a limit of saturated ideals, then $I^{\sat}\cdot I_L$, where $I_L$ is the ideal of $L$, is contained in $I$.
\end{proposition}

For a linear space $L \subseteq T_1$ we denote by $L^{\perp} \subseteq S_1$
its perpendicular (this notion has nothing to do with $q$!). Recall also that
the linear maps $q:S_1\to T_1$ and
$q^{-1}:T_1\to S_1$ are nondegenerate and inverses of each other.

\begin{lemma}\label{quadricPolarityLemma}
Let $L, N \subset T_1$ be linear subspaces with $\dim L + \dim N = \dim T_1$. Then the following are equivalent
\begin{enumerate}
    \item $L^\perp\cdot N^\perp \subseteq q^\perp$,
    \item $q(L^\perp\cdot N^\perp) = 0$,
    \item $q(L^\perp) = N$,
    \item $q^{-1}(N) = L^\perp$,
    \item $q^{-1}(N\cdot L) = 0$,
    \item $N \cdot L \subseteq (q^{-1})^{\perp}$.
\end{enumerate}
\end{lemma}
\begin{proof}
The equivalences follows immediately from the definitions when noticing that $q(L^\perp)(N^\perp)=q(L^\perp\cdot N^\perp)$.
\end{proof}

\begin{corollary}[a geometric condition for being in the $\VPSsbl$]\label{kri}
In the notation of Proposition~\ref{prop:necessaryForLimitOfSat} assume that $I\subseteq q^\perp\subset S$. Let $L\subset T_1$ be the line in $V(I_2)$  and let  $N\subset T_1$ be the linear span of $V(I^{\sat})$. Then $q^{-1}(L\cdot N) = 0$.
\end{corollary}
\begin{proof}
We have $I^{\sat}\cdot I_L \subseteq I \subseteq q^\perp$. The span $N$ is
equal to $(I^{\sat})_1^{\perp}$ and $L$ is equal to $(I_{L})_1^{\perp}$. By the assumption of Proposition \ref{prop:necessaryForLimitOfSat}, we have $\dim N + \dim L = \dim T_1$, so the claim follows from Lemma~\ref{quadricPolarityLemma}.
\end{proof}

\begin{corollary}\label{linesinquadric} Let $[I]\in \VPSsbl, n\leq 5$ be an unsaturated limit ideal. Then $\Gamma=V(I)$ contains a subscheme of length  $4$ in a line $L$, and the line $L$ is contained in the inverse quadric $Q^{-1}.$  In particular, when $n=4$,  $I=I_\Gamma\cap q^\perp$.
\end{corollary}
\begin{proof}  Let $I=I_\Gamma\cap q^\perp$, where, of course, $I_\Gamma=I^{\rm sat}$.
Then the linear span of $N=V(I_\Gamma)\subset T_1$ contains $L$.
But then $L^2\subset N\cdot L$, and hence, by Corollary \ref{kri},
$q^{-1}(L^2)=0$, which means $L\subset Q^{-1}$. When $n=4$, $I_\Gamma$ has
Hilbert function $(1,2,3,4,4, \ldots )$ while $I_\Gamma\cap q^\perp$ has
Hilbert function $(1,4,4,4,4, \ldots )$, so $I=I_\Gamma\cap q^\perp$.
\end{proof}
A finite apolar scheme of minimal length is Gorenstein.  This property does not necessarily hold for unsaturated limit ideals.  
\begin{example}[Non-Gorenstein unsaturated limit ideal for $n=5$]\label{ex:nonGorenstein}
    Let $q = y_1y_4 + y_2y_3 + y_5^2$. Consider the family over $\mathbb{A}^1
    = \Spec \CC[t]$ given by the ideal
    \[
        \left(x_{4}x _{5},\,x_{3}x_{5},\,x_{1}x_{5},\,x_{4}^{2
        },\,x_{3}x_{4},\,x_{1}^2t + x_{2}x_{4},\,x_{1}x_{4}-x_{
        5}^{2},\,x_{3}^{2},\,x_{2}x_{3}-x_{5}^{2},\,
        x_{1}x_{3},\,x_{1}^{4}\right)
    \]
    It is contained in $q^\perp$ and the fiber over $t=\lambda\neq 0$ is
    saturated. The fiber over $t = 0$ is non-saturated and abstractly
    isomorphic to $\Spec \CC[\varepsilon_1, \varepsilon_2]/(\varepsilon_1^4,
    \varepsilon_1\varepsilon_2, \varepsilon_2^2)$, which in particular is not
    Gorenstein.
\end{example}

Recall that $\Satbar^H \subseteq \Hilb^H$ is the union of irreducible components of
$\Hilb^H$ and it is defined as the closure of the locus of saturated $[I]\in \Hilb^H$.
The following theorem allows us to relate smoothness
a point $[I]\in \VPSsbl$ and the same point $[I]\in \Satbar^H$. This is useful both theoretically
and in computations, see cases $n=4,5$ below.
\begin{proposition}\label{ref:smoothness}
    For every $\GL_n$-stable locally closed subscheme $\mathcal{Z}$ of
    $\Hilb^H$ the map
    \[
        \GL_n \times (\VPS\cap \mathcal{Z})\to \mathcal{Z};\quad (g,I)\mapsto g\cdot I
    \]
    is smooth and has $\binom{n+1}{2}$-dimensional fibers.
\end{proposition}
\begin{proof}
    Smoothness of the map follows by Theorem~\ref{thm:orthogonalization} and infinitesimlar lifting criterion~\cite[\href{https://stacks.math.columbia.edu/tag/02H6}{Tag 02H6}]{stacks-project}. {For the fiber, by $\GL_n$-equivariance it is enough to compute the dimension of the fiber over any point $I\in \VPS\cap \mathcal{Z}$. The fiber consists of pairs $(g, g^{-1}I)$ such that $g^{-1}I\subset q^{\perp}$. The containment is equivalent to $I\subset (g\cdot q)^{\perp}$. Now, we have $H(2) = n$, so there is an $n$-dimensional space of quadrics $q'$ such that $I\subset (q')^{\perp}$. The fiber is isomorphic to the set of $g\in \GL_n$ which map $q$ into this space. The action of $\GL_n$ on quadrics is transitive, so the dimension of the fiber is $n + \binom{n}{2} = \binom{n+1}{2}$ as claimed.}
\end{proof}

\begin{corollary}[being a limit of saturated is independent of being in $\VPS$]\label{ref:saturableAndApolarImpliesSaturableInApolar:cor}
    We have $\VPSsbl = \VPS \cap \Satbar^H$ as {schemes}. For $n\leq
    13$, we have $\VPSsbl = \VPS \cap \Slip^H$ as schemes.
\end{corollary}
\begin{proof}
    By definition, we have $\VPSsbl \subseteq \VPS \cap \Satbar^H$ as schemes. {To prove the other inclusion, pick a point $[I]\in \VPS\cap \Satbar^H$ and let $A = \hat{\OO}_{\Satbar^H, [I]}$ be the complete local ring of $[I]\in \Satbar^H$. Since $A$ is complete, by applying Theorem~\ref{thm:orthogonalization} to finite order truncations of the map $\phi\colon\Spec(A)\to \Satbar^H$ we obtain a pair of maps $g\colon \Spec(A)\to \GL_n$ and $\psi\colon\Spec(A)\to \VPS\cap \Satbar^H$ with $\phi = g\cdot \psi$. Pick any generic point $\eta\in\Spec(A)$. Since $\Satbar^H$ is the closure of saturated locus, the ideal $\phi(\eta)$ is saturated. We have $\phi(\eta) = g(\eta)\cdot \psi(\eta)$, so also the ideal $\psi(\eta)$ is saturated. This ideal is apolar to $Q$, thus $\psi(\eta)$ lies in $\VPSgood$. This holds for every generic point $\eta$, so $\psi$ factors through the closure of $\VPSgood$, that is, through $\VPSsbl$. This proves the other containment from the statement.}
    For $n\leq 13$ all Gorenstein algebras are smoothable, {hence the locus $\VPSgood$ is contained {scheme-theoretically} in $\Slip^H$ and so its closure $\VPSsbl$ is contained in $\Slip^H$. This yields (still for $n\leq 13$) a chain of inclusions of {schemes}
    \[
    \VPSsbl \subseteq \VPS \cap \Slip^H \subseteq \VPS\cap \Satbar^H = \VPSsbl,
    \]
    which then have to be equalities.
    }
\end{proof}
\begin{corollary}\label{unsaturatedlimits4}
Let $\Gamma$ be a scheme of length $4$ contained in a line in $Q^{-1}$ and let $I=I_\Gamma\cap q^\perp$, then $I$ is an unsaturated limit ideal in $\VPSsbl$.
\end{corollary}
\begin{proof} $I$ is an unsaturated limit ideal by \cite[Proposition 4.2]{JM} and it is apolar to $Q$, hence lies in $\VPSsbl$ by Corollary \ref{ref:saturableAndApolarImpliesSaturableInApolar:cor}.
\end{proof}
\subsection{Unsaturated limit ideals in the case $n=5$}

With reference to \cite[\S4.4.2 and \S4.4.3]{JM}, we will describe $\VPSsbl\cap \VPSuns,$ when $n=5$.

For $[I]\in \VPS
\setminus \VPSgood$, by Proposition~\ref{prop:regularity}, the
length five scheme $\Gamma \subseteq \PP^4$ contains a length four subscheme
$\Gamma'$ contained in a line.
It follows that $\VPSuns$ is a union of two families $\mathcal{F}_1,
\mathcal{F}_{2}$, described as follows.

A general point of $\mathcal{F}_1$ is an
ideal $I$ such that $V(I^{\sat}) = \Gamma' \sqcup \{p\}$, where $p$ is a point
and $\Gamma'$ has length four and is contained in a line. For general such
$\Gamma'$ and $p$, the Hilbert function of $S/I^{\sat}$ is $(1,3,4,5,5,
\ldots)$, so $I = I^{\sat} \cap q^\perp$; the ideal $I$ is determined by
$I^{\sat}$ as in the case of $n=4$. The family $\mathcal{F}_1$ is irreducible and has dimension
$6+4+4$.

For the family $\mathcal{F}_{2}$ a general point is an
ideal $I$ such that $\Gamma = V(I^{\sat})$ is a length five scheme on a line. The ideal $I$ is obtained from $I^{\sat}$ as the
intersection $I = I^{\sat} \cap c^{\perp} \cap q^\perp$ where $c$ is a cubic form such that $H_{S/I}(2) = 5$. We do not know
whether the family $\mathcal{F}_2$ is irreducible or what is its dimension.

Let us now consider the saturable elements of both families. It follows from
Corollaries~\ref{kri}-\ref{linesinquadric} and from~\cite[\S4.4.2]{JM} that $\mathcal{F}_1 \cap
\VPSsbl$ consists exactly of ideals $I$ such that the line $L = \langle \Gamma'\rangle$ is
contained in $Q^{-1}$ and $p$ is contained in the plane polar to $L$, here $p$
is such that $\Supp \Gamma = \{p\} \cup \Supp \Gamma'$ counting multiplicities. The set of unsaturated limit ideals in $\mathcal{F}_1$  has dimension $3+4+2$.
In the case $\mathcal{F}_2$, again by Corollary~\ref{linesinquadric} the line
$\langle \Gamma\rangle$ must be contained in $Q^{-1}$. Moreover
by~\cite[\S4.4.3]{JM}, the cube $c$
must be of the form $\ell^2\cdot \mu$, where $\ell$ is a linear form on the
line and $\mu$ is a linear form on the polar plane. The set of unsaturated limit ideals in $\mathcal{F}_2$ has
dimension $3 + 5 + 1$.
\begin{remark}\label{contraction}
When $n=4,5$ and $I$ is an unsaturated limit ideal, then the map $$\pi_G: {\VPSsbl}\setminus\VPSgood\to{\VPSsbl}_G;\quad [I]\mapsto [I_2]$$ is a forgetful map with positive dimensional fibers. The boundary ${\VPSsbl}\setminus \VPSgood$ is a divisor by Proposition \ref{boundarydivisors}, and has dimension $5$ and $9$ when $n=4$ and $n=5$ respectively. The ideal $I_2$ vanishes on the line that contains a subscheme of length four (or five) of $V(I^{\rm sat}).$    In cases $n=4$ and $n=5$ with $I\in \mathcal{F}_1$, the ideal $I_2$ depend only on a line, respectively a line and a point, so the family of ideals $I_2$ in the image of $\pi_G$ has dimension one and five respectively.  In case  $n=5$ and  $I\in \mathcal{F}_2$, the image $I_2\in{\VPSsbl}_G$, depend on the line and a special cubic form, so the family of ideals $I_2$ in the image of $\pi_G$ has dimension at most $6$.
\end{remark}

\section{The schemes $\VPSsbl$ and $\VPSsbl_G$ for $n\leq 5$}\label{globalproperties}

As above, we let $H = (1,n,n, \ldots)$.  
It follows from Proposition~\ref{prop:regularity} that for $n\leq 3$ the schemes
$\VPSsbl$ and $\VPSsbl_G$ are both equal to $\VPSgood$ and hence smooth. We show
they are both smooth also for $n=4,5$. For $n\geq 6$, already $\VPSgood$ is singular.

The idea of the proof is to prove that $\Slip^H$ is smooth in the unsaturated
limit ideals of ${\VPSsbl}$, then use Proposition~\ref{ref:smoothness} and Corollary~\ref{ref:saturableAndApolarImpliesSaturableInApolar:cor} for $\mathcal{Z} = \Satbar^H = \Slip^H$ to deduce the same for ${\VPSsbl}$ itself and
finally prove the smoothness for their projections in the Grassmannian.
We begin by showing that $\Slip^H$ is smooth at
the unsaturated limit ideals of $\VPS$.
\begin{proposition}\label{prop:SlipSmooth}
    Let $n=4,5$ and consider a homogeneous ideal $[I]\in \Slip^H$ such that
    $H_{S/I^{\sat}} = (1,n-2, n-1, n, n,  \ldots )$. Then $[I]$ is a smooth point on $\Slip^H$.
\end{proposition}
\begin{proof}
    We will make an upper bound on the dimension of the tangent space to
    $\Slip^H$ at $[I]$. Unfortunately, this space has no functorial
    interpretation, hence we begin with the tangent space $T_{[I]}\Hilb^H
    \simeq \Hom_S(I, S/I)_{0}$.
    The short exact sequence $0\to I^{\sat}/I\to S/I\to S/I^{\sat}$ yields
    \[
        0\to \Hom_S(I, I^{\sat}/I)_0\to \Hom_S(I, S/I)_0\to
        \Hom_S(I, S/I^{\sat})_0.
    \]
    We have $(I^{\sat}/I)_3 = 0$, while $(I^{\sat}/I)_2$ is one-dimensional, so  
    $$\dim_{k}\Hom_S(I, I^{\sat}/I)_0=\dim_{k} I_2 =
    \binom{n}{2},$$
    and hence $$\dim_{k} T_{[I]}\Hilb^H\leq \dim_k\Hom_S(I, S/I^{\sat})_0+\binom{n}{2}.$$
    The short exact $0\to I\to I^{\sat}\to I^{\sat}/I\to 0$ yields
    \[
        0 \to \Hom_S(I^{\sat},
        S/I^{\sat})_0\to \Hom_S(I, S/I^{\sat})_0 \to \Ext^1(I^{\sat}/I,
        S/I^{\sat})_0.
    \]
    Using Proposition~\ref{prop:regularity} and \cite[Example~4.1]{JM}, we obtain
    that
    \[
        \dim_{k} \Ext^1(I^{\sat}/I, S/I^{\sat})_0 = 1.
    \]
    Since $n-2\leq 3$,
    by~\cite[Corollary~3.16]{JM} we obtain that the point $[I^{\sat}]\in
    \Hilb^{H_{S/I^{\sat}}}$ is smooth. The saturated locus of this Hilbert
    scheme maps isomorphically to its image in $\Hvanilla$,
    see~\cite[Proposition~3.9]{JM}, hence has dimension 
    $$\dim_k\Hom_S(I^{\sat},
        S/I^{\sat})_0=(n-1)(n-4) + 2(n-2)
    + 4 = (n-1)n - 2n+4,$$ where $(n-1)(n-4)$ comes from a choice of $n-4$
    points, $2(n-2)$ a choice of a line and $4$ from a choice of four points
    on it.
    In total we obtain a bound
    \begin{equation}\label{eq:tangentBound}
        \dim_{k} T_{[I]}\Hilb^H \leq (n-1)n - (2n-4) + \binom{n}{2} + 1.
    \end{equation}
    We now pass to $T_{[I]}\Slip^H$. The inclusion
    $\Hom_S(I, I^{\sat}/I)_{0} \subseteq \Hom_S(I, S/I)_0$ corresponds to the
    tangent map for the embedding of $\varphi^{-1}([I^{\sat}])\subseteq \Hilb^H$,
    where $\varphi\colon \Hilb^H\to \Hvanilla$ is the natural map.

    By Proposition~\ref{prop:necessaryForLimitOfSat} we see that the intersection
    $\varphi^{-1}([I^{\sat}]) \cap \Slip^H$ consists of $I$ such that
    $$I^{\sat}_1 \cdot (I_{L})_1 \subseteq I_2,$$ where $L$ is the line in $V(I^{\sat}_2)$. This is true even
    scheme-theoretically and shows that the intersection above is the
    Grassmannian $\Gr{1}{I^{\sat}_2/(I^{\sat}_1 \cdot (I_{L})_1)}$, which has
    dimension $\dim_{k} I^{\sat}_2 - 1 - \dim_{k} I^{\sat}_1 \cdot (I_{L})_1
    = \binom{n}{2} - (2n-5)$. Arguing as above, the
    estimate~\eqref{eq:tangentBound} yields
    \[
        \dim_{k} T_{[I]} \Slip^H \leq (n-1)n - (2n-4) + \binom{n}{2} + 1 -
        (2n-5) = (n-1)n + \frac{1}{2}(n-4)(n-5).
    \]
    The variety $\Slip^H$ is birational to the smoothable component of
    $\Hvanilla$, hence has dimension $(n-1)n$. This proves that
    $[I]\in \Slip^H$ is smooth for $n=4,5$.
\end{proof}

\begin{proposition}\label{prop:SlipSmoothSecondCase}
    Let $n=5$ and consider a homogeneous ideal $[I]\in \Slip^H$ such that
    $H_{S/I^{\sat}} = (1, 2, 3, 4, 5, 5, \ldots )$. Then $[I]$ is a smooth point on $\Slip^H$.
\end{proposition}
\begin{proof}
    The strategy of the proof is the same as in
    Proposition~\ref{prop:SlipSmooth}. As in that proposition, we argue that
    the tangent space at $[I^{\sat}]\in \Hilb^{(1,2,3,4, \ldots ,5)}$ is
    $11$-dimensional because the locus of saturated ideals with this Hilbert
    function is parameterized by $6$-dimensional choice of line and
    $5$-dimensional choice of five points. It follows that $T_{[I]}\Hilb^H$ is
    an extension of the at most $12$-dimensional space $V := \Hom_S(I, S/I^{\sat})$ by
    $\Hom_S(I, I^{\sat}/I)$. The latter space is the tangent space to the
    fiber $\pi^{-1}([I^{\sat}])$ of $\pi\colon \Hilb^H\to \Hvanilla$. This fiber is smooth and
    rational of dimension $14$, its parameterization is given
    in~\cite[\S4.4.3]{JM}. Also by~\cite[\S4.4.3]{JM} the intersection
    $\pi^{-1}([I])\cap \Slip^H$ is smooth and rational of dimension $8$. It
    follows that inside the $26$-dimensional space $T_{[I]}\Hilb^H$ the
    intersection of
    $T_{[I]} \Slip^H$ with $T_{[I]}\pi^{-1}([I^{\sat}])$ is $8$-dimensional.
    The latter space is $14$-dimensional, hence the dimension of $T_{[I]}\Slip^H$ is at most
    $26 - 14 + 8 =20$, which coincides with the dimension of $\Slip^H$, so equality holds and $|I]$ is a smooth point on $\Slip^H$.
\end{proof}

\begin{corollary}\label{ref:VSPborderSmooth:cor}
    For $n=4,5$ the variety $\VPSsbl$ is smooth.
\end{corollary}
\begin{proof}
    Being smooth is preserved by smooth maps. It follows that the good part
    $\VPSgood$ is smooth by Proposition~\ref{smoothMap}. The variety
    $\VPSsbl$ is smooth at unsaturated limit ideals  by
    Proposition~\ref{ref:smoothness} {with Corollary~\ref{ref:saturableAndApolarImpliesSaturableInApolar:cor}} and
    Propositions~\ref{prop:SlipSmooth}-\ref{prop:SlipSmoothSecondCase}.
\end{proof}
Some natural simplified versions of Proposition~\ref{prop:SlipSmooth} are
false: for example $\Hilb^H$ can be singular at the points as in
the statement and
$\Slip^H$ can be singular away from the locus in the Proposition; both pathologies occur for $n=4$.
It would be interesting to know whether Proposition~\ref{prop:SlipSmooth}
holds for higher $n$.

\begin{remark}\label{ref:VSPinHilb:remark}
   {\rm We do not know whether $\VSPn\subset \Hvanilla$ is smooth even for $n=4,5$. The map $\varphi\colon \Slip^H\to
    \Hvanilla$ restricts to a map $\VPSsbl\to \VSPn$ which is an isomorphism
    on $\VPSgood$. For $n=4$ this map is also bijective on points, because for an
    unsaturated
    limit ideal $I$ we have $I = I^{\sat} \cap q^\perp$. This shows
    that for $n=4$ the variety $\VPSsbl$ is the normalization of $\VSPn$.}
\end{remark}

Now we prove that also the Grassmannian compactification
$$\VPSsbl_G\subset \Gr{\binom{n}{2}}{q^\perp_2}\subset \Gr{\binom{n}{2}}{S_2}$$ is smooth for $n=4,5$.
One cannot hope to have Proposition~\ref{ref:smoothness} for $\VPSsbl_G$ and
$\Gr{\binom{n}{2}}{S_2}$ because the latter lacks structure. But there is an
analogue of the Hilbert scheme in this setup. We define the \emph{syzygetic
locus} $\Syz\subseteq \Gr{\binom{n}{2}}{S_2}$ as the closed subscheme given by the
determinantal equations which on a point $[V]\in \Gr{\binom{n}{2}}{S_2}$ boil
down to $H_{S/(V)}(3) \geq n$, where $(V)\subset S$ is the ideal generated by $V\subset S_2$.
The tangent space to $[V]\in \Syz$ at a point where $H_{S/(V)}(3) = n$ is given by
\[
    \Hom_S\left((V) + S_{\geq 4},\, \frac{S}{(V)+S_{\geq 4}}\right)_0.
\]
{The name \emph{compactification} is justified by the following.
\begin{proposition}\label{prop:goodlocusInGrassmannian}
    The map $\pi_G\colon \Hgood\to \Syz$, given by $\pi_G(I) = I + S_{\geq 4}$ is an open immersion.
    In particular, the map $$\pi_G\colon\VPSgood\to \VPSsbl_G$$ is an isomorphism onto $\VPSgood_G$.
\end{proposition}
\begin{proof}
    A point of $\Hgood$ corresponds to a saturated ideal $I$ such that the quotient by $I$ has Hilbert function $H$. It follows that the regularity of the quotient is one and the regularity of $I$ is two. By~\cite[Proposition~3.1]{erman_Murphys_law_for_punctual_Hilb} the complete local rings of $[I]\in \Hgood$ and $\pi_G([I])\in \Syz$ are isomorphic. It follows~\cite[\href{https://stacks.math.columbia.edu/tag/039M}{Tag 039M}]{stacks-project}  that $\pi_G$ is \'etale at $[I]$. From the regularity it also follows that $I$ is generated by $I_2$, so $\pi_G$ is injective on points. It follows~ \cite[\href{https://stacks.math.columbia.edu/tag/025G}{Tag 025G}]{stacks-project} that this map is an open immersion. The result on $\VPSgood$ follows by intersecting with the locus apolar to $Q$.
\end{proof}}

\begin{proposition}\label{prop:GrassSlipSmooth}
    Let $n=4$, let $[I]\in \VPSsbl$ be an unsaturated limit ideal and let $[I_2]=\pi_G([I])\in \VPSsbl_G$. Then $[I_2]$ is a smooth
    point of the image of $\Slip^H$ in $\Gr{6}{S_2}$.
\end{proposition}
\begin{proof}
    Every unsaturated limit ideal $I$ defines a scheme $\Gamma=V(I)$ of length four on a line on a quadric by
    Proposition~\ref{prop:necessaryForLimitOfSat} and
    Corollary~\ref{linesinquadric}. The ideal $I$ itself is determined by $\Gamma$
    as $I=I_\Gamma\cap q^\perp$.
    Using a torus action not changing the quadric, we may degenerate
    further so that $\Gamma$ is supported only at a single point.
    There is then, up to coordinate change, only one unsaturated limit ideal. If $q = y_1y_3 +
    y_2y_4$, then we may take it as $I := (x_2^4, x_3, x_4)\cap q^\perp$.

    Using the package \emph{VersalDeformations}~\cite{Ilten}, see~Appendix~\ref{sec:computerPart},
    we check that near the point $[I_2]\in \Syz$ the scheme $\Syz$ is reduced
    and it has two irreducible components $\mathcal{Z}_1$, $\mathcal{Z}_2$ passing through $[I_2]$; the
    dimensions are {$12$ and $10$}, respectively. Moreover, both these components
    are smooth at $[I_2]$. The image of {$\Slip^H$} under $\pi_G\colon \Hilb^H \to
    \Gr{6}{S_2}$ is an integral {$12$}-dimensional variety, hence it
    has to coincide with the larger component: $\pi_G({\Slip^H}) =
    \mathcal{Z}_1$. In particular, {$[I_2]$ is its smooth point}.
\end{proof}

\begin{proposition}\label{prop:GrassSlipSmoothFive}
    Let $n=5$ and let $[I]\in \VPSsbl \cap \VPSuns$ be an unsaturated limit ideal. Then its image $[I_2]=\pi_G([I])\in \VPSsbl_G$ is a smooth
    point of the image of $\Slip^H$ in $\Gr{10}{S_2}$.
\end{proposition}
\begin{proof}
    The argument is analogous to Proposition~\ref{prop:GrassSlipSmooth},
    although reducing to finitely many possible limits requires much more work
    (we need finitely many of them to compute the tangent space dimensions).
    By Corollary~\ref{linesinquadric} and
    Proposition~\ref{prop:regularity},
    an unsaturated limit ideal $I$ corresponds to a scheme
    $\Gamma=V(I^{\sat})$ of length five with a subscheme $\Gamma_0\subset
    \Gamma$ of length four on a line $L\subseteq Q^{-1}$ and a possibly
    embedded fifth point. Up to coordinate change, we may
    assume that $q = y_1y_3 + y_2y_4 + y_5^2$ and the line $L$ is $x_1 = x_4 = x_5
    = 0$. We have $q^{-1}=4x_1x_3 + 4x_2x_4 + x_5^2$.
    Consider two cases for $\Gamma = V(I^{\sat})$:
    \begin{enumerate}
        \item\label{it:curvilinear} the scheme $\Gamma$ is contained in $L$,
        \item\label{it:notcurvilinear} the scheme $\Gamma$ is not contained in $L$.
    \end{enumerate}
    We begin with the case~\ref{it:notcurvilinear}. By Corollary~\ref{kri} the fifth point
    lies on the plane $\Pi = (x_1 = x_4 = 0)$.
    Consider a coordinate change
    \begin{align*}
        x_5 &\mapsto x_5 + \lambda_4x_4 +
        \lambda_1 x_1,\\
        x_3&\mapsto x_3 - \frac{1}{4}\lambda_1^2x_1 - \frac{1}{4}\lambda_1\lambda_4 x_4
        - \frac{1}{2}\lambda_1x_5,\\
        x_2&\mapsto x_2 - \frac{1}{4}\lambda_4^2x_4-\frac{1}{4}\lambda_1\lambda_4
        x_1- \frac{1}{2}\lambda_4x_5.
    \end{align*}
    It preserves $Q^{-1}$ and $L$ and maps a point $[0:0:0:0:1]\in \Pi$ to
    $[\lambda_1:0:0:\lambda_4:1]$. Using such an action, we may assume that
    the fifth point is not $[0:0:0:0:1]$.
    Consider the torus action by
    \[
        t\circ[y_1:y_2:y_3:y_4:y_5] = [ty_1:ty_2:ty_3:ty_4:t^2y_5].
    \]
    For $t\to 0$ the fifth point converges to a point of $L$.
    The set of singular points of the image is closed, so we may reduce to
    the case of the fifth point on a line; either the point is embedded or we
    reduce to case~\ref{it:curvilinear}. We consider the
    case~\ref{it:curvilinear} below, so assume that the point is embedded.

    For every $\lambda\in \CC$ the
    coordinate change $x_2\mapsto x_2 + \lambda x_3$, $x_1 \mapsto x_1 - \lambda
    x_4$ maps $Q^{-1}$, $L$ and $\Pi$ to themselves.
    Using such an action for general $\lambda$, we may assume that $\Gamma$ is
    not supported at $[0:1:0:0:0]$.
    Consider the torus action by
    \[
        t\circ[y_1:y_2:y_3:y_4:y_5] = [y_1:ty_2:y_3:t^{-1}y_4:y_5].
    \]
    Again it preserves $Q^{-1}$, $L$ and $\Pi$. The limit with $t\to 0$ of any
    point $$[0:y_2:y_3:0:0]\in L \setminus \{[0:1:0:0:0]\}\;{\rm is}\;[0:0:1:0:0].$$ Hence, we reduce
    to the case of $\Gamma$ supported at $[0:0:1:0:0]$ having an embedded
    point, so that $\Gamma_0=(x_1=x_2^4=x_4=x_5=0)$ and $\Gamma$ spans $\Pi$. Here $I$ is determined by $I =
    I(\Gamma) \cap q^{\perp}$. We check its smoothness
    in~\S\ref{sec:computerPart}. (This is the case from
    Example~\ref{ex:nonGorenstein}.)
    
    The case~\ref{it:curvilinear} is similar, but slightly more algebraic,
    as in this case the coordinate changes need to take into account the
    structure of $I$ rather than $I^{\sat}$. By~\cite[Proposition~4.4(3)]{JM},
    the space $I^{\perp}_3$ is spanned by $y_2^3, y_2^2y_3, y_2y_3^2, y_3^2$
    and $y_5\cdot \ell^2$, where $\ell\in \CC[y_2, y_3]_1$.
    Making the coordinate change as above, we may assume that $\Gamma =
    V(I^{\sat})$ is not supported at $[0:1:0:0:0]$ and that $\ell$ is not
    proportional to $y_2$. Then a torus degeneration as above reduces us to
    the case $y_5y_3^2$ and $\Gamma$ supported at $[0:0:1:0:0]$, thus $I$ is
    determined as $I = I(\Gamma) \cap (y_5y_3^2, q)^{\perp}$. Again, we check its smoothness
    in~\S\ref{sec:computerPart}.
\end{proof}

\begin{remark}
    Actually, according to the calculation in~\S\ref{sec:computerPart}, near
    $[I_2]$ the variety $\pi_G(\Hilb^H)$ is equal to $\Syz$ for $n=4,5$.
\end{remark}

\begin{proposition}\label{G-smooth}
    Let $\pi_G\colon \Hilb^H\to \Syz$ be the natural map.
    The map $$\GL_n\times \VPSsbl_G\to \pi_G(\Slip^H);\quad (g,I_2)\mapsto g\cdot I_2$$ is smooth, so in
    particular $\VPSsbl_G$ is smooth for $n=4, 5$.
\end{proposition}
\begin{proof}
    We have $\pi_G(\VPSsbl) = \VPSsbl_G$ and the map $\pi_G$ sends $I$ to $I_2$,
    so the smoothness of the map follows by
    Theorem~\ref{thm:orthogonalization}.
    {For $n\leq 5$ the locus $\VPSgood$ is smooth so $\pi_G(\VPSgood)$ is smooth by Proposition~\ref{prop:goodlocusInGrassmannian}.}
    {Moreover, the scheme} $\VPSsbl$ is smooth at the unsaturated limit ideals by
    Propositions~\ref{prop:GrassSlipSmooth}-\ref{prop:GrassSlipSmoothFive} {and by the smoothness of the map $\GL_n\times \VPSsbl_G\to \pi_G(\Slip^H)$. Hence $\VPSsbl_G$ is smooth.}
\end{proof}

When $n=1$ and $n=2$ the variety $\VPSsbl_G$ is $\PP^1$, respectively a $3$-fold
general linear section of $\Gr{2}{5}$, cf.~\cite{Mukai}. 
   In particular, in both cases, the Grassmannian subscheme $\VPSsbl_G$ is a Fano variety of index $2$.   The same holds for $n=4,5$.

\begin{corollary}\label{ref:Fano:cor}
   $\VPSsbl_G$ is a Fano variety of index $2$, when $n=4,5$.
\end{corollary}
\begin{proof}
We follow the argument of \cite[Theorem 6.1]{RS}.
Consider the set
$$H_h=\{[I_2]\in \VPSsbl_G\ |\ V(I_2)\cap h\not= \emptyset\}$$
where $h\subset \PP(T_1)$ is a hyperplane.  
In \cite[Lemma 4.9]{RS} the set $H_h$ is defined for schemes $\Gamma\in
{\rm VPS}(Q,n)$, but the definition extends naturally.  Similarly, the proof of the lemma extends to show that $H_h$ is the restriction of a Pl\"ucker divisor on $\Gr{\binom{n}{2}}{S_2}$.
Notice that the divisors $H_h$ contains the image of all unsaturated limit
ideals, when $n=4,5$, since $V(I_2)$ in these cases contains a
line that necessarily intersects the hyperplane $h$.  By \cite[Proposition
5.11]{RS}, the complement of $H_h$ is isomorphic to an affine space, while $H_h$
is very ample by Proposition \ref{G-smooth}. The image $BL\subset \VPSsbl_G$ of
the set of all unsaturated limit ideals is not a divisor, so the argument of
the proof of \cite[Theorem 6.1]{RS} applies: Since the complement of $H_h$ is
isomorphic to an affine space, the Picard group of $\VPSsbl_G$ is $\ZZ$ as soon as the restriction $H_h$ of the special Pl\"ucker divisor is irreducible. So assume $H_h=H_1+H_2$. 

Any apolar scheme $\Gamma$ lies in the complement of some $h$, so the subset $BL$ is the base locus of the special divisors $H_h$.   But $BL$ is not a divisor, so  $H_1$ and $H_2$ must both move with base locus contained in $BL$. 

Now $H_h\cdot l=1$ for every line $l\subset \VPSsbl_G\setminus BL$, so only one of the two components can have positive intersection with $l$. In particular, one of the components, say $H_2$ has intersection $H_2\cdot l=0$ and contains every line $l$ that it intersects. 
By, \cite[Lemma 6.2]{RS}, any two smooth apolar schemes $\Gamma, \Gamma'\in
\VPSsbl_G\setminus BL$ are connected by a sequence of lines, so $H_2$ would
contain all of $\VPSsbl_G$.  Therefore $H_h$ is irreducible.

The computation of the Fano-index in \cite[Theorem 6.1]{RS} when $n=4,5$,
applies similarly to $\VPSsbl_G$, see also Remark~\ref{index2}.

\end{proof}

\subsection{The scheme $\VPSsbl_G$ for $n=4$}
In the case $n=4$, we give a more precise description of the Grassmannian compactification, starting with the image $BL\subset \VPSsbl_G$
 of the family of unsaturated limit ideals in $\VPSsbl$ by the map
$\pi:\VPSsbl\to \Syz$. By Corollaries \ref{linesinquadric} and
\ref{unsaturatedlimits4} each unsaturated limit ideal is contained in the
ideal of a line in $Q^{-1}$, and the degree two part of such an ideal depends
only on the line: it is the space of those quadrics in the ideal of the line that lie in $q^\perp$. Therefore $$BL=\VPSsbl_G\cap\VPSuns_G=C_1\cup C_2$$ is the disjoint union of two rational curves $C_1$ and $C_2$, one for each pencil of lines in $Q^{-1}$.

\begin{lemma}\label{degree6lem} Let $n=4$. The two curves $C_1$ and $C_2$ are rational normal curves of degree $6$ in 
$\VPSsbl_G\subset \Gr{6}{q^\perp_2}$.
\end{lemma}
\begin{proof} For each line $L\subset Q^{-1}$, we let $I_L$ be the intersection of its ideal with the space of quadrics in the apolar ideal $q^\perp$.  
Thus $I_L$ is $6$-dimensional in $q^\perp$, and its orthogonal is a $3$-dimensional subspace $I_L^\perp\subset (q^\perp_2)^*$.
We proceed by arguing in $\Gr{3}{(q^\perp_2)^*}$.  Projectively  the subspace $\PP(I_L^\perp)$ is the span of the conic  $v_2(L)\subset v_2(Q^{-1})$ under the veronese embedding $v_2: \PP(T_1)\to \PP(T_2)$. The rational $\PP^2$-scroll $X$ formed by these spans as $L$ moves in one pencil of lines in $Q^{-1}$ is a rational normal scroll: For each triple of lines in one pencil in $Q^{-1}$, the corresponding conics on $X$ is not contained in any hyperplane, so the degree of $X$ is at least $6$, the minimal degree of a $3$-fold in $\PP((q^\perp_2)^*)$.  On the other hand, $v_2(Q^{-1})$ is linearly normal, hence also $X$, so $X$ has degree $6$.  The Grassmannian embedding of a rational normal scroll of planes $X$ in the Grassmannian $\Gr{3}{(q^\perp_2)^*}$ of planes is a rational normal curve and has the same degree as $X$.  The corresponding embedding in $\Gr{6}{q^\perp_2}$ of the  pencil of spaces $\PP(I_L)$ dual to $\PP(I_L^\perp)$ is projectively equivalent and therefore also have degree $6$.
\end{proof}

\begin{lemma}\label{ref:normalBundle:lem}
    For each rational normal curve $C_i, i=1,2$ in $\VPSsbl_G$, the
    tangent bundle
    of $\VPSsbl_G$ restricted to $C_i$ is $\OO(2)^{\oplus 6}$ and the normal bundle is $\OO(2)^{\oplus 5}$.
\end{lemma}
\begin{proof}
    Fix coordinates so that $Q = y_1y_4 + y_2y_3$ and $C$ is the rational
    normal curve containing the point $p = [(x_3, x_4)\cap q^\perp]$.
    Explicitly, we have
    \begin{equation}\label{eq:ideal}
        (x_3, x_4)\cap q^\perp = (x_1x_3, x_2x_3 - x_1x_4, x_3^2, x_2x_4, x_3x_4,
        x_4^2).
    \end{equation}
    Consider the $\SL_2$-action on $k[x_1, \ldots ,x_4]$ where a matrix
    $[a_{ij}]\in \SL_2$ acts by
    \[
        \begin{pmatrix}
            a_{00} & 0 & a_{10} & 0\\
            0 & a_{00} & 0 & -a_{10}\\
            a_{01} & 0 & a_{11} & 0\\
            0 & -a_{01} & 0 & a_{11}
        \end{pmatrix}
    \]
    The quadric $Q$ is invariant under this action and so $\SL_2$ acts on
    $\VPS_G$ and on $C$. Actually, the action on $C$ is transitive; the
    point $p$ is sent to $(x_1a_{01} + x_{3}a_{11}, -x_2a_{01}+x_4a_{11})\cap
    q^\perp$; thus the induced map $\SL_2\to \Aut(C)  \simeq \PGL_2$ is the
    usual one. The action above induces $\SL_2$-linearizations on tangent bundle and normal bundle from the statement.
    The action of the standard torus of $\SL_2$ decomposes the degree two part
    of~\eqref{eq:ideal} into weight spaces
    \[
        \langle x_1x_3, x_2x_3 - x_1x_4, x_2x_4 \rangle,\quad \mbox{and}\quad
        \langle x_3^2, x_3x_4, x_4^2\rangle
    \]
    of weights $0$ and $-2$, respectively.
    The tangent space $T_{p}\VPS_G$ is equal to
    \[
        \Hom_{S}\left( \left((x_3, x_4)\cap q^\perp\right) + S_{\geq 4},\ \frac{q^\perp}{(x_3,
            x_4)\cap q^\perp + S_{\geq 4}} \right)_0,
    \]
    which is $9$-dimensional. Let $W = \langle x_1^2, x_1x_2, x_2^2\rangle\subset q^\perp$, it is a weight space of weight two.
    Any linear map $\rho_2$ from the weight zero space in $(x_3, x_4)\cap q^\perp$ to  
    $W$ has weight two.  To see that $\rho_2$ lifts to a tangent vector $\rho$ it is enough to check that it induces a linear map $\rho_3$ in degree three that is zero on the linear syzygies of $(x_3, x_4)\cap q^\perp$, since the linear syzygies generate all syzygies of $(x_3, x_4)\cap q^\perp$. Since $\rho_2$ has weight two,  it induces, in degree three, a map $\rho_3$ of the weight one space $\langle x_1,x_2 \rangle\cdot\langle x_1x_3, x_2x_3 - x_1x_4, x_2x_4 \rangle$  to a weight three space of cubics.  
  
    But any linear syzygy of $(x_3, x_4)\cap q^\perp$ has weight $\pm 1$ so extending $\rho_2$ to the zero map on the weight $(-2)$ space $\langle x_3^2, x_3x_4, x_4^2\rangle$, means that  $\rho_3$ is zero on the syzygies.  We conclude that $\rho_2$ lifts to a tangent vector $\rho$.
    
    The space of maps $\rho_2$ is
    $9$-dimensional, hence equal to $T_{p}\VPS_G$, the tangent space to $\VPS_G$ at $p$. It
    follows that the weights of the action of the standard torus of $\SL_2$ on
    $T_{p}\VPS_G$ are all equal to $2$.
    The fiber of the tangent bundle from the statement is a subrepresentation
    of $T_{p}\VPS_G$ for this torus, so it also has all weights equal to $2$.
    The fiber of the normal bundle is its quotient representation, so the same
    holds here. By the classification of vector bundles on $\mathbb{P}^1$, we
    find that the tangent bundle is $\OO(2)^{\oplus 6}$, while the normal
    bundle is $\OO(2)^{\oplus 5}$.
\end{proof}

\begin{remark}\label{degreecomp} {\rm The degree and normal bundle of the curves $C_1$ and $C_2$ is also found in the Macaulay2 computation \cite{Mac2S}.
}
\end{remark}
\begin{remark}\label{index2} {\rm The normal bundle of the two curves $C_1\cup C_2=\VPSsbl_G\cap \VPSuns_G$ may be used to compute the index of the latter variety as a Fano variety. The Picard group is generated by $H$, so the canonical divisor is  $K=-aH$ for some integer $a$.   Since the normal bundle  in $\VPSsbl_G$ of each of the two rational curves $C_i$ in $BL$ has degree $10$ while the degree of
each curve is $6$, the adjunction formula implies $-2={\rm deg}K_{C_i}= 10-{\rm deg}(aH\cap C_i)$, so $a=2$ and $\VPSsbl_G$ has Fano index $2$.}

\end{remark}
\begin{proposition}\label{ref:linearSection:prop} Let $n=4$. The subscheme
    $\VPSsbl_G$ inside $\Gr{6}{q^{\perp}_2}$ is an arithmetically Cohen-Macaulay variety and a smooth linear section
$$\VPSsbl_G=\PP^{38}\cap \Gr{6}{q^{\perp}_2},$$
of degree $362$.
The image of the set of unsaturated limit ideals $\VPSsbl_G\cap \VPSuns_G$ form a union $C_1\cup C_2$ of two disjoint rational normal curves each of degree $6$.

A general codimension $6$ linear space containing $C_1\cup C_2$ intersects the locus $\VPSgood_G$ in $310$ points.
\end{proposition}
\begin{proof}
It remains to prove the linear section and Cohen-Macaulay properties and the degree computations.  This involves the Macaulay2 computation \cite{Mac2S}, see Theorem \ref{appendix:B1} in Appendix \ref{vsp4package}. 
It proves that the linear section $X:=\PP^{38}\cap \Gr{6}{q^{\perp}_2}$, where $\PP^{38}$ is the linear span of $\VPSsbl_G$, is a $6$-dimensional arithmetically Cohen-Macaulay variety of degree $362$.

Now, we argue that the degree of $X$ coincides with the degree of $\VPSsbl_G$, so that $\VPSsbl_G$ is the only component of maximal dimension in $X$.  But, since $X$ is arithmetically Cohen-Macaulay, every component is maximal, so $X$ and $\VPSsbl_G$ coincide if their degree does.

The degree argument uses excess intersection.  Let $L$ be a general linear space of codimension $6$ that contain the two curves $C_1\cup C_2$.  The excess intersection of the two curves in the intersection $L\cap \VPSsbl_G$ is computed by
the formula \cite[Prop 9.1.1.(2)]{Fulton}:
$${\rm deg}(L\cdot \VPSsbl_G)^{C_i}={\rm deg}(c_1(N_L))-c_1(N_{C_i})=6(H\cdot C_i)-10=6^2-10,$$
where $N_L$ is the normal bundle of $L$ restricted to $C_i$ and $H$ is the class of a hyperplane.
So the excess intersection along $C_1\cup C_2$ has degree $2(6^2-10)=52$. 
By the below remark, the linear space $L$ intersects $$\VPSgood_G=\VPSsbl_G\setminus (C_1\cup C_2)$$ in $310$ points. So the intersection of  $\VPSsbl_G$ with a general linear space of codimension $6$ is $310+52=362$, which coincides with the computation of the degree of $X$.
\end{proof}

\begin{remark}\label{excess}  {\rm
In \cite{RS} special Pl\"ucker hyperplanes are considered in a computation of $\VPSsbl_G$. They already appeared in the proof of  Corollary \ref{ref:Fano:cor}.
For a hyperplane $h=V(l)\subset \PP(T_1)$, the set
$$H_h=\{[I_2]\in \VPSsbl_G\ |\ V(I_2)\cap h\not= \emptyset\}$$
is a  Pl\"ucker hyperplane section of  $\VPSsbl_G$.
Since $V(I_2)$ contains a line for every unsaturated limit ideal $I$, every divisor $H_h$ contains the image of all unsaturated limit
ideals.
According to \cite[Theorem 6.3]{RS}, when $n=4$, the intersection of six general hyperplanes $H_h$ contains $310$ ideals of
polar simplices, i.e. points on $\VPSgood_G$.  That theorem claims that this is also the degree of $\VPSgood_G$, based on the false assumption that this variety is closed. The locus $$\VPSsbl_G\setminus \VPSgood_G=C_1\cup C_2$$ which is contained in every hyperplane $H_h$.
The intersection of six general hyperplane sections $H_h$ therefore contains $C_1\cup C_2$ in addition to the $310$ points on $\VPSgood_G$. So to compute the degree of $\VPSsbl_G$ using the formula in \cite[Theorem 6.3]{RS}, becomes the excess intersection problem of finding the contribution of $C_1\cup C_2$ to the degree of the intersection.
}
\end{remark}

\appendix
\section{Computer code}\label{sec:computerPart}

In this section we exhibit the computer code used to prove smoothness of
specific points in $\VPSsbl_G$ for $n=4,5$. The computations are performed
over $\QQ$. 
See
package~\cite{Ilten} for details.
\begin{verbatim}
loadPackage("VersalDeformations", Reload=>true);
computeComponents = I -> (
    -- check that Syz and Hilb agree
    assert(hilbertFunction(3, I) == dim ring I);
    J = ideal select(flatten entries mingens I, x->sum degree x<=2);
    -- the ideal I_2 + S_{\geq 4}
    Jtr = ideal mingens(J + (ideal gens ring J)^4);
    tg = normalMatrix(0, Jtr);
    reducedtg = CT^1(0, Jtr);
    ob = CT^2(0, Jtr);
    (F,R,G,C) = versalDeformation(mingens Jtr,reducedtg,ob);
    IG = ideal mingens ideal sum G;
    pd = primaryDecomposition IG;
    assert(#pd == 2);
    -- actually, those are linear spaces
    assert(dim ideal singularLocus pd_0 == -1);
    assert(dim ideal singularLocus pd_1 == -1);
    use ring I; -- reset the default ring
    return (rank source tg - codim pd_0, 
                            rank source tg - codim pd_1);
);
\end{verbatim}
The code for $n=4$ is as follows.
\begin{verbatim}
QQ[x_0 .. x_3];
Q = x_0*x_2 + x_1*x_3;
Qperp = inverseSystem(Q);
I = intersect(Qperp, ideal(x_1^4, x_2, x_3)); --case n=4
computeComponents(I) == (12,10) -- true
\end{verbatim}
The code for $n=5$ is very similar.
\begin{verbatim}
QQ[z_0 .. z_4];
Qz = z_0*z_2 + z_1*z_3 + z_4^2;
Qzperp = inverseSystem(Qz);
I = ideal(z_3*z_4, z_2*z_4, z_0*z_4, z_3^2, z_2*z_3, z_1*z_3,
    z_0*z_3-z_4^2, z_2^2, z_1*z_2 - z_4^2, z_0*z_2, z_0^4);
computeComponents(I) == (20,20) -- true
I = intersect(ideal(z_0, z_3, z_4, z_1^5),
    inverseSystem(z_4*z_2^2), inverseSystem(Qz));
computeComponents(I) == (20,20) -- true
\end{verbatim}
The code for $n=4$ is as follows.
\begin{verbatim}
QQ[x_1 .. x_4];
Q = x_1*x_3 + x_2*x_4;
Qperp = inverseSystem(Q);
I = intersect(Qperp, ideal(x_2^4, x_3, x_4)); --case n=4
computeComponents(I) == (12,10) -- true
\end{verbatim}
The code for $n=5$ is very similar.
\begin{verbatim}
QQ[x_1 .. x_5];
Qz = x_1*x_3 + x_2*x_4 + x_5^2;
Qzperp = inverseSystem(Qz);
I = ideal(x_4*x_5, x_2*x_5, x_1*x_5, x_4^2, x_1*x_4,x_3*x_4, 
    2*x_2*x_4-x_5^2, x_1^2,x_1*x_2, 2*x_1*x_3 - x_5^2,  x_2^4);
computeComponents(I) == (20,20) -- true
I = intersect(ideal(x_1, x_4, x_5, x_2^5),
    inverseSystem(x_5*x_3^2), inverseSystem(Qz));
computeComponents(I) == (20,20) -- true
\end{verbatim}

\section{Remarks on the package vsp4.m2}\label{vsp4package}

In our package we prove computationally the following theorem.
The computations are performed over $\QQ$.

 \begin{theorem}\label{appendix:B1} Consider $\VPSsbl$ with $H=(1,4,4,...)$. The Grassmannian model  
 $$\VPSsbl_G \subset \Gr{6}{q^{\perp}_2} \subset \PP^{\binom{9}{6}-1}$$
 is smooth. Then $\VPSsbl_G$ is the intersection of $\Gr{6}{q^{\perp}_2}$ with the linear span of
 $\VPSsbl_G$ inside  $\PP^{\binom{9}{6}-1}$.
 It is an arithmetically Cohen-Macaulay variety of degree $362$ with h-vector
 $(1,32,148,148,32,1)$.
 \end{theorem}

The proof is computational.
The first step consists in computing the
unfolding~\cite{Hauser_Sings} of the ideal
$\left({x}_{0}^{2},{x}_{0}{x}_{1},{x}_{1}^{2},{x}_{0}{
     x}_{2}-{x}_{1}{x}_{3},{x}_{1}{x}_{2},{x}_{0}{x}_{3},{x
     }_{2}^{4}\right)$
and the computation of the flatness relations via a Gr\"obner basis computation. As it turns out the base space of the family
consists of two components of dimension $8$ and $6$ respectively. The $8$-dimensional family consists of lines in $\PP^3$ together with 4 points. The general element of the 6-dimensional family consists of 4 distinct points not on a line, they vary as sets of points apolar to $q$.
The intersection of these families consists of 4 points on a line on the inverse quadric $Q^{-1}$. So the intersection has two components corresponding to the two family of lines on $Q^{-1}$. They are of dimension $5$.

In the next step we compute the image of these families in the Grassmannian  $\Gr{6}{q^{\perp}_2}$. It is computational
accessible to compute the linear equations in $\PP^{\binom{9}{6}-1}$. 
To see that image of the 8-dimensional family coincides with the $3$-uple embedding of $\Gr{2}{4} \subset \PP^5$ is straightforward. 

Analyzing the image of the $6$-dimensional family is computationally difficult.  
The linear span is a $\PP^{38}$ and the 1050 Pl\"ucker quadrics restrict to $380$ independent quadrics on this $\PP^{38}$.
Moreover, the quadrics define a saturate ideal $J$ of dimension $7$ and degree $362$. So $J$ corresponds to a 6-dimensional  scheme $X$ of degree $362$. The 6-th difference function of the Hilbert function of $X$ takes values $(1,32,148,148,32,1)$.

This lead us to conjecture that $X$ is an arithmetically Gorenstein subscheme of $\PP^{38} \subset \PP^{\binom{9}{6}-1}$.
Note that $\binom{31+2}{2} = 148 +380$. Thus if $X$ is arithmetically
Gorenstein, its homogeneous ideal in $\PP^{38}$ is  generated by these 380 quadrics. However we could not check directly that the homogeneous ideal of the image of the 6-dimensional family is generated by quadrics.

Instead we establish that $X$ is arithmetically Cohen-Macaulay by testing that for a carefully chosen sparse sequence 
$\ell_1,\ldots,\ell_6$ of linear forms the ideals $J+(\ell_1,\ldots,\ell_i)$ are still saturated for $i=1 \ldots,6$.
Since $\dim (J+(\ell_1,\ldots,\ell_6))=1$ (which we check computationally) the
ideal  $J+(\ell_1,\ldots,\ell_6)$ defines the homogeneous ideal  of a zero dimensional subscheme of a $\PP^{32}$ of degree $362$. Thus there exists a final linear form $\ell_7$ such that $\ell_1,\ldots,\ell_7$ are a regular sequence for the homogeneous coordinate ring of $X$. 

This proves that $X$ is arithmetically Cohen-Macaulay and therefore unmixed.

In Proposition \ref{ref:linearSection:prop} above, we apply excess intersection theory to show that the two varieties $X$ and $\VPSsbl_G$ have the same degree, hence must coincide.

By the Macaulay2 computation, the image of the intersection of the two families consists of two rational normal curves $C_1,C_2$ of degree $6$, as above.  Furthermore it checks that $X$ is smooth along the normal curves $C_i \cong \PP^1$ and  that 
their normal bundles are isomorphic to $\oplus^5_{1} \OO_{\PP^1}(2)$, in accordance with Lemmas \ref{degree6lem} and \ref{ref:normalBundle:lem}.
Finally, by inspecting the $6 \times 9$ matrix which defines the map of the 6-dimensional family into the Grassmannian we see that $\VPSsbl_G$ is smooth outside $C_1 \cup C_2$.


\begin{thebibliography}{HJNY21}

\bibitem[Art76]{artin_deform_of_sings}
Michael Artin.
\newblock {\em Deformations of singularities}.
\newblock Notes by C.S. Seshadri and Allen Tannenbaum. Tata Institute of
  Fundamental Research, Bombay, India, 1976.

\bibitem[BB14]{BB10}
Weronika Buczy\'{n}ska and Jaros{\l}aw Buczy\'{n}ski.
\newblock Secant varieties to high degree veronese reembeddings, catalecticant
  matrices and smoothable gorenstein schemes.
\newblock {\em Duke Math. J.}, 23(1):63--69, 2014.

\bibitem[BB21]{BB}
Weronika Buczy\'{n}ska and Jaros{\l}aw Buczy\'{n}ski.
\newblock Apolarity, border rank, and multigraded {H}ilbert scheme.
\newblock {\em Duke Math. J.}, 170(16):3659--3702, 2021.

\bibitem[BM18]{Bolognesi_Massarenti}
Michele Bolognesi and Alex Massarenti.
\newblock Varieties of sums of powers and moduli spaces of {$(1,7)$}-polarized
  abelian surfaces.
\newblock {\em J. Geom. Phys.}, 125:23--32, 2018.

\bibitem[CCO17]{Carlina_Catalisano_Oneto}
Enrico Carlini, Maria~Virginia Catalisano, and Alessandro Oneto.
\newblock Waring loci and the {S}trassen conjecture.
\newblock {\em Adv. Math.}, 314:630--662, 2017.

\bibitem[CJN15]{cjn13}
Gianfranco Casnati, Joachim Jelisiejew, and Roberto Notari.
\newblock Irreducibility of the {G}orenstein loci of {H}ilbert schemes via ray
  families.
\newblock {\em Algebra Number Theory}, 9(7):1525--1570, 2015.

\bibitem[CN09]{casnati_notari_irreducibility_Gorenstein_degree_9}
Gianfranco Casnati and Roberto Notari.
\newblock On the {G}orenstein locus of some punctual {H}ilbert schemes.
\newblock {\em J. Pure Appl. Algebra}, 213(11):2055--2074, 2009.

\bibitem[Erm12]{erman_Murphys_law_for_punctual_Hilb}
Daniel Erman.
\newblock Murphy's law for {H}ilbert function strata in the {H}ilbert scheme of
  points.
\newblock {\em Math. Res. Lett.}, 19(6):1277--1281, 2012.

\bibitem[Ful98]{Fulton}
William Fulton.
\newblock {\em Intersection theory}, volume~2 of {\em Ergebnisse der Mathematik
  und ihrer Grenzgebiete. 3. Folge. A Series of Modern Surveys in Mathematics
  [Results in Mathematics and Related Areas. 3rd Series. A Series of Modern
  Surveys in Mathematics]}.
\newblock Springer-Verlag, Berlin, second edition, 1998.

\bibitem[GRV18]{Gallet_Ranestad_Villamizar}
Matteo Gallet, Kristian Ranestad, and Nelly Villamizar.
\newblock Varieties of apolar subschemes of toric surfaces.
\newblock {\em Ark. Mat.}, 56(1):73--99, 2018.

\bibitem[GS]{M2}
Daniel Grayson and Michael Stillman.
\newblock Macaulay2, a software system for research in algebraic geometry, ver.
  1.21.
\newblock Available at \url{https://macaulay2.com}.

\bibitem[Hau83]{Hauser_Sings}
Herwig Hauser.
\newblock An algorithm of construction of the semiuniversal deformation of an
  isolated singularity.
\newblock In {\em Singularities, {P}art 1 ({A}rcata, {C}alif., 1981)},
  volume~40 of {\em Proc. Sympos. Pure Math.}, pages 567--573. Amer. Math.
  Soc., Providence, R.I., 1983.

\bibitem[HJNY21]{hoyois2021hermitian}
Marc Hoyois, Joachim Jelisiejew, Denis Nardin, and Maria Yakerson.
\newblock Hermitian k-theory via oriented gorenstein algebras, 2021.

\bibitem[HMV20]{Huang_Michalek_Ventura}
Hang Huang, Mateusz Micha{\l}ek, and Emanuele Ventura.
\newblock Vanishing {H}essian, wild forms and their border {VSP}.
\newblock {\em Math. Ann.}, 378(3-4):1505--1532, 2020.

\bibitem[HS04]{Haiman_Sturmfels__multigraded}
Mark Haiman and Bernd Sturmfels.
\newblock Multigraded {H}ilbert schemes.
\newblock {\em J. Algebraic Geom.}, 13(4):725--769, 2004.

\bibitem[Ilt12]{Ilten}
Nathan~Owen Ilten.
\newblock Versal deformations and local {H}ilbert schemes.
\newblock {\em J. Softw. Algebra Geom.}, 4:12--16, 2012.

\bibitem[JM22]{JM}
Joachim Jelisiejew and Tomasz Ma{\'{n}}dziuk.
\newblock Limits of saturated ideals.
\newblock 2022.
\newblock arXiv:2210.13579.

\bibitem[Muk92]{Mukai}
Shigeru Mukai.
\newblock Fano 3-folds.
\newblock {\em London Math. Soc., Lect. Notes Ser.}, 179:255--263, 1992.

\bibitem[RS00]{ranestad_schreyer_VSP}
Kristian Ranestad and Frank-Olaf Schreyer.
\newblock Varieties of sums of powers.
\newblock {\em J. Reine Angew. Math.}, 525:147--181, 2000.

\bibitem[RS13]{RS}
Kristian Ranestad and Frank-Olaf Schreyer.
\newblock The variety of polar simplices.
\newblock {\em Doc. Math.}, 18:469--505, 2013.

\bibitem[RV17]{Ranestad_Voisin}
Kristian Ranestad and Claire Voisin.
\newblock Variety of power sums and divisors in the moduli space of cubic
  fourfolds.
\newblock {\em Doc. Math.}, 22:455--504, 2017.

\bibitem[Sch23]{Mac2S}
Frank-Olaf Schreyer.
\newblock vsp4, a {M}acaulay2 package for the computation of the variety of
  polar simplices of a quadric in {$P^3$}.
\newblock
  \url{https://www.math.uni-sb.de/ag/schreyer/index.php/computeralgebra}, 2023.

\bibitem[SSS22]{Skjelnes_Stahl}
Roy~Mikael Skjelnes and Gustav S{\ae}d\'{e}n~St{\aa}hl.
\newblock Explicit projective embeddings of standard opens of the {H}ilbert
  scheme of points.
\newblock {\em J. Algebra}, 590:254--276, 2022.

\bibitem[{Sta}23]{stacks-project}
The {Stacks project authors}.
\newblock The stacks project.
\newblock \url{https://stacks.math.columbia.edu}, 2023.

\end{thebibliography}
\end{document}